\documentclass[reqno,english]{amsart}
\usepackage[utf8]{inputenc}
\usepackage{amsmath,amssymb,amsthm,mathrsfs,color,times,textcomp,yfonts,mathtools,cases}
\usepackage[left=0.8in, right=0.8in, top=0.8in, bottom=0.8in]{geometry}

\newcommand{\Sph}{\mathbb S}

\allowdisplaybreaks[4]

\usepackage{bm}

\usepackage[T1]{fontenc} %加粗

\usepackage{subcaption}
\usepackage[normalem]{ulem}
\usepackage[export]{adjustbox}
\usepackage{esint}
\usepackage{xcolor}
\usepackage{array}
\usepackage[colorlinks=true]{hyperref}
\hypersetup{urlcolor=blue, citecolor=blue, linkcolor=red}

\hypersetup{
colorlinks=true,
linkcolor=red
}
\usepackage{indentfirst}
\usepackage{graphicx}
\usepackage{float}
\numberwithin{equation}{section}
\newtheorem{theorem}{Theorem}[section]
\newtheorem{lemma}[theorem]{Lemma}
\newtheorem{corollary}[theorem]{Corollary}
\newtheorem{proposition}[theorem]{Proposition}

\newtheorem{remark}{Remark}[section]

\newcommand{\Rom}[1]{%
\textup{\uppercase\expandafter{\romannumeral#1}}%
}

\newcommand*{\ud}{d}
\newcommand{\R}{\mathbb R}
\newcommand{\Id}{\operatorname{Id}}

\usepackage{mathtools}
\usepackage{thmtools}
\declaretheoremstyle[headfont=\normalfont]{normalhead}
\usepackage{color}
\usepackage{morefloats}
\usepackage{textcomp}
\usepackage{url}
\title[Axially Symmetric Rigidity on $\Sph^N$]
{Axially Symmetric Rigidity for a $Q$-Curvature-Type Equation on $\Sph^N$}

\author{Changfeng Gui}
\address{Department of Mathematics, University of Macau, Taipa, Macau}
\email{changfenggui@um.edu.mo}

\author{Tuoxin Li}
\address{Department of Mathematics \\  The Chinese University of Hong Kong \\ Shatin \\ NT \\ Hong Kong}
\email{txli@math.cuhk.edu.hk}

\author{Juncheng Wei}
\address{Department of Mathematics \\  The Chinese University of Hong Kong \\ Shatin \\ NT \\ Hong Kong}
\email{wei@math.cuhk.edu.hk}

\author{Zikai Ye}
\address{Department of Mathematics \\  The Chinese University of Hong Kong \\ Shatin \\ NT \\ Hong Kong}
\email{zkye@math.cuhk.edu.hk}
\begin{document}

\subjclass{35B07, 35J15, 35J35, 35J60, 53A05, 53C18, 53C21} 

\keywords{Q-curvature-type equation, Beckner's inequality, rigidity, Gegenbauer polynomials, GJMS operator,}

\begin{abstract}
We prove that for every even integer $N\geq 4$ or $N\in\{3,5,7\}$, every axially symmetric solution to the $Q$-curvature-type problem
$$ \alpha P_N u + (N-1)!(1-\frac{e^{Nu}}{\int_{\mathbb{S}^N} e^{Nu}dw})=0 \ \ \ \ \ \mbox{on} \ \mathbb{S}^N $$
is constant, provided that $ \alpha\ge\frac{1}{2}$ and $\alpha \not =1$. The proof starts from a Gegenbauer expansion and its weighted $\ell ^2$ estimate on coefficients. Then one key estimate is the refined estimate of the seminorm $\lfloor G\rfloor^2$ using a specially designed integration-by-parts identity for general even integer $N\geq 4$. For $N\in\{3,5,7\}$, the seminorm estimates  are established via a descent to a three--dimensional nonlocal difference quotient identity. 
\end{abstract}

\maketitle
\hypersetup{linkcolor=black}
\tableofcontents

\section{Introduction and main results}

\subsection{A Q-curvature-type equation}

For $\alpha>0$, we consider the $Q$-curvature-type equation on $\mathbb{S}^N$:
\begin{equation}\label{paneitz}
\alpha P_N u+(N-1)!(1-\frac{e^{Nu}}{\int_{\mathbb{S}^N}e^{Nu}dw})=0,
\end{equation}
where $dw$ denotes the normalized surface measure on $\mathbb{S}^N$ so that $\int_{\mathbb{S}^N}dw=1$. Here
\begin{equation*}
\begin{aligned}
    P_N:=
    \begin{cases}
        \prod_{k=0}^{\frac{N-2}{2}}(-\Delta+k(N-k-1)),&\text{ for }N\text{ even}, \\
        \left(-\Delta+(\frac{N-1}{2})^2\right)^{\frac{1}{2}}\prod_{k=0}^{\frac{N-3}{2}}(-\Delta+k(N-k-1)),&\text{ for }N\text{ odd}
    \end{cases}
\end{aligned}
\end{equation*}
represents the Paneitz operator on $\mathbb{S}^N$ introduced by Paneitz in \cite{Paneitz2008}.

The Paneitz operator $P_N$ belongs to the family of conformally invariant operators $P_{N,g,k}$ called the GJMS operator, constructed by Graham, Jenne, Mason, and Sparling on a Riemannian manifold $(M^N, g)$ in \cite{GJMS1992}. When $k=\frac{N}{2}$, its covariance law is
\begin{equation*}
P_{N,\tilde g,}(\phi)=e^{-N\omega}P_{N,g}(\phi),\qquad \text{ when }\tilde g=e^{2\omega}g
\end{equation*}
We refer to \cite{CLY2019, CM2023, ChangYang1995, ChangYang1997, DHL2000,DM2008, FG2013, GHX2021, GurMal2015, LiXiong2019, Mal2006, WX1998} and the references therein for results and background on $Q$-curvature problems and GJMS operators.  

Equation \eqref{paneitz} has a variational structure. In fact, it is the Euler-Lagrange equation of the functional
\begin{equation}\label{eq:Beckner-functional}
 \mathcal J_{\alpha,N}(u)
 :=\frac{\alpha}{2}\int_{\Sph^N}uP_Nu\,\ud w
 +(N-1)!\int_{\Sph^N}u\,\ud w
 -\frac{(N-1)!}{N}
  \log\!\left(\int_{\Sph^N}e^{Nu}\,\ud w\right).
\end{equation}
The well-known Beckner's inequality \cite{Beckner1993} states that $\inf \mathcal{J}_{1,N}(u)=0$.
In addition, A. Chang and P. Yang \cite{ChangYang1995} proved that, if $u$ belongs to the zero center-of-mass set
\begin{equation*}
\mathcal{L}_{N}=\left\{u\in H^{\frac{N}{2}}(\mathbb{S}^N)\ :\ \int_{\mathbb{S}^N}e^{Nu} x_j dw=0,\ j=1,\cdots, N+1 \right \},
\end{equation*}
then for any $\alpha\geq \frac{1}{2}$, there exists a constant $C(\alpha,N)\geq0$ such that 
\begin{equation*}
\mathcal{J}_{\alpha,N}(u)\geq -C(\alpha,N)    
\end{equation*}
for any $u\in \mathcal{L}_{N}$. This motivates the study of the constrained Euler--Lagrange equation of $\mathcal{J}_{\alpha,N}$ in the constraint set $\mathcal{L}_N$.

A critical point of $\mathcal{J}_{\alpha,N}$ under the zero center-of-mass constraint satisfies the following $Q$-curvature-type equation on $\mathbb{S}^N$:
\begin{equation*}
\alpha P_N u+(N-1)!(1-\frac{e^{Nu}}{\int_{\mathbb{S}^N}e^{Nu}dw})=\sum_{i=1}^{N+1}a_i x_i e^{Nu} \ \mbox{on} \ \mathbb{S}^N
\end{equation*}
for some constants $a_i$, $i=1,\dots, N+1$. Interestingly, for $\alpha\neq1$, by applying the Kazdan–Warner identity, one can show that all Lagrange multipliers $a_i$ vanish. 
The converse is also true. In fact, by \cite[Proposition 1.5]{GHW2022}, we have 
\begin{equation}\label{u centre}
\int_{\mathbb{S}^N} u x_jdw=\int_{\mathbb{S}^N}  e^{Nu}x_jdw=0, \quad j=1,\cdots, N+1.    
\end{equation}
We refer to Chang-Yang \cite{ChangYang1995} and Wei-Xu \cite{WX1998} for more details. 

\medskip

On $\mathbb{S}^2$, \eqref{paneitz} is known as the Nirenberg problem:
\begin{equation*}
-\alpha \Delta u + 1-  \frac{e^{2u}}{\int_{\mathbb{S}^2} e^{2u}}=0 \ \ \mbox{on} \ \mathbb{S}^2.
\end{equation*}
This problem has been extensively studied over the past four decades. See \cite{ChangYang1987, ChangYang1988, JinLiXiong2017} and the references therein. For $\alpha<1$, Feldman, Froese, Ghoussoub, and the first author  \cite{FFGG1998} first established the rigidity for axially symmetric functions when $\alpha >\frac{16}{25}-\epsilon$. Here we say a function $u$ on $\mathbb{S}^2$ is axially symmetric if, up to a rotation, $u=u(x)$ for $x=x_{1}\in (-1,1)$. The first and the third authors \cite{GW2000} subsequently proved the sharp axially symmetric version.
%Later, Ghoussoub and Lin \cite{GL2010} showed that all critical points of $J_{\alpha,2}$ are axially symmetric and hence the conjecture holds true for $\frac{2}{3}-\epsilon<\alpha<1$. 
Finally, the first author and Moradifam \cite{GM2018} established the sphere covering inequality to prove that every solution is  axially symmetric if $\alpha\ge \frac{1}{2}$, thereby completing the classification on $\mathbb{S}^2$. For more result in this direction, we refre readers to \cite{ChangGui202, ChangHang2022, GHM2020, SSTW2019}
%Furthermore, Shi, Sun, Tian and Wei \cite{SSTW2019} showed that all even solutions are axially symmetric for $\frac{1}{4}\leq \alpha<1$. For more general results on improved Moser-Trudinger-Onofri inequalities on $\mathbb{S}^2$ and their connections with the Szeg\"o limit theorem, see \cite{ChangGui202, ChangHang2022}.

For the higher-dimensional equation \eqref{paneitz} on $\mathbb{S}^N$, as mentioned above, the third author and Xu proved the rigidity when $\alpha<1$ is close to $1$. Since direct treatment of \eqref{paneitz} seems difficult, in view of the work of Ghoussoub-Lin \cite{GL2010} and  the work of the first author and Moradifam \cite{GM2018}, it is natural to consider axially symmetric functions as a first step in higher dimensions. Several results have been obtained for axially symmetric solutions with $N=4,6,8$. We refer readers to \cite{GHX2021,GHW2022,LWY2022,GLWY2025}.
\medskip

For odd $N$, there are few results on solutions of \eqref{paneitz}.  When $N=1$, using the Szeg\"o limit theorem (see \cite{Szego1975}), Widom \cite{Widom1988} improved the best constant of Beckner's inequality under extra orthogonality conditions. Consequently, for $ \alpha \geq \frac{1}{m+1}$, \eqref{paneitz} has only trivial solutions if $u\in H^\frac{1}{2}(\mathbb S^1)$ satisfies $\int_{\mathbb S^1}u\ud \theta=0$ and $\int_{\mathbb S^1}e^{u}e^{ik\theta}\ud \theta=0$, for $k=1,\dots,m$.

For odd $N\geq 3$, as mentioned before, Wei-Xu \cite{WX1998} proved the rigidity when $\alpha$ is close to $1$. Recently, Zhang \cite{Zhang2025} proved the rigidity for $\alpha>1$. To the best of our knowledge, there are no other results. One of the difficulties is the non-locality of the GJMS operator $P_N$ when $N$ is odd.

In this paper, we focus on axially symmetric solutions on $\mathbb{S}^N$. We prove the rigidity in the axially symmetric case for every even integer $N\geq 4$ and $N=3,5,7$.

As derived in \cite{GHW2022}, for even $N$ and an axially symmetric function $u=u(x)$, equation \eqref{paneitz} reduces to
\begin{equation}\label{axial}
\alpha(-1)^{\frac{N}{2}}[(1-x^2)^{\frac{N}{2}}u']^{(N-1)}+(N-1)!-\frac{(N-1)!\sqrt{\pi}\Gamma(\frac{N}{2})}{\Gamma(\frac{N+1}{2})\gamma}e^{Nu}=0,\ x\in(-1,1),
\end{equation}
where
\begin{equation*}
\gamma=\int_{-1}^{1}(1-x^2)^\frac{N-2}{2} e^{Nu} \ud x.
\end{equation*}
%The constraint set becomes
%\begin{equation}\label{AxialLrN}
%\mathcal{L}_{r,N}=\left\{u\in H^{\frac{N}{2}}(\mathbb{S}^N):\ u=u(x)\text{ and }\int_{-1}^{1}x(1-x^2)^{\frac{N-2}{2}}e^{Nu} d x=0\right\},
%\end{equation}

Our first main result establishes rigidity for axially symmetric functions for even dimensions $N\geq 4$.

\begin{theorem}\label{main}
Let $N\geq 4$ be an even integer. If $\alpha\ge \frac{1}{2}$ and $\alpha\neq 1$, then all solutions of \eqref{axial}  are constant functions.
\end{theorem}

\medskip

For odd $N$ and an axially symmetric function $u=u(x)$, the GJMS operator $P_N$ can be written as
\begin{equation}\label{def P rN}
P_{r,N}=\left(-(1-x^2)\frac{d^2}{dx^2}+Nx\frac{d}{dx}+(\frac{N-1}{2})^2\right)^{\frac{1}{2}}\prod_{k=0}^{\frac{N-3}{2}}\left(-(1-x^2)\frac{d^2}{dx^2}+Nx\frac{d}{dx}+k(N-k-1)\right).
\end{equation}
In this case, equation \eqref{paneitz} reduces to
\begin{equation}\label{axialodd}
\alpha P_{r,N}u+(N-1)!-\frac{(N-1)!\sqrt{\pi}\Gamma(\frac{N}{2})}{\Gamma(\frac{N+1}{2})\gamma}e^{Nu}=0,\ x\in(-1,1).
\end{equation}

For $N=3,5,7$, we have the following result.
\begin{theorem}\label{mainodd}
Let $N\in\{3,5,7\}$. If $\alpha\ge \frac{1}{2}$ and $\alpha\neq 1$, then all solutions of \eqref{axialodd}  are constant functions.
\end{theorem}

\begin{remark}
    When $\alpha=1$, there is a family of solutions $u(x)=-\log(1-ax)$ to \eqref{paneitz} for any $a\in(-1,1)$. Then we see that \eqref{u centre} does not hold unless $a=0$, and this is why we have to exclude $\alpha=1$ in our theorems. In addition, this family of solutions are also optimizers of Beckner's inequality.
\end{remark}

\medskip

\subsection{Main ideas of the proof}
As mentioned before, Gui-Hu-Xie \cite{GHX2021, GHW2022} partially proved Theorem \ref{main} for $N=4 ,6,8$ and a non-optimal lower bound of $\alpha$  via a strategy similar to that in \cite{GW2000}. Li-Wei-Ye \cite{LWY2022} and Gui-Li-Wei-Ye \cite{GLWY2025} closed the gaps and reached the best constant $\alpha=\frac{1}{2}$ for $N=4 ,6,8$. In all of those works, an auxiliary function $G=(1-x^2)u'$ is expanded  in terms of Gegenbauer polynomials,  and  a quantity $D$ is introduced in relation to the Gegenbauer coefficients and the estimate of a seminorm $\lfloor G\rfloor^2$ (see \eqref{Gfloor def}). In particular, \cite{GLWY2025} observed a new way of integration-by-parts for
\begin{equation*}
I_N=(-1)^{\frac{N}{2}}\int_{-1}^{1}(1-x^2)^{\frac{N-2}{2}}G^2[(1-x^2)^{\frac{N-2}{2}}G]^{(N-1)} dx
\end{equation*}
with $N=6,8$ to estimate $\lfloor G\rfloor^2$. 

However, the integration-by-parts identities in the previous papers were constructed separately in dimension $6$ and $8$. No dimension-independent strategy was found. For general $N$, our first objective is to write $I_N$ in the form 
\begin{equation*}
    I_N=\sum_{r=1}^{\frac{N}{2}-1}\sum_{d=0}^{\frac{N}{2}-1-r}C_{r,d}\int_{-1}^{1}(1-x^2)^{\frac{N}{2}-1+r}(G^{(r)})^2\hat G_{\frac{N}{2}-1-r-d} dx,
\end{equation*}
where $\hat{G}_j$'s is a family of specially-designed auxiliary function that admits pointwise upper bound derived in Proposition \ref{Gfloor est}. Dropping $\hat{G}_j$ with the pointwise upper bound yields refined estimates of $I_N$, hence the seminorm $\lfloor G\rfloor^2$.

However, as $N$ grows, $I_N$ contains derivatives of increasingly higher orders, leading to higher complexity of the integration-by-parts. Furthermore, the choice of $C_{r,d}$ is observed to be non-unique. We need to choose suitable $C_{r,d}$ so that the corresponding seminorm estimate is strong enough. These difficulties are overcome in this paper by a series of carefully chosen $C_{r,d}$ and a detailed operator algebra computation of the integration by parts procedure. This is a key ingredient of the paper.

On the other hand, previous papers \cite{GHX2021, GHW2022, GLWY2025, GW2000, LWY2022} used the lower bound of $D_N$ to generate a series of inequalities and proved $a=\frac{N}{N+1}(1-\alpha\beta)\leq \frac{d_0}{\lambda_n}$ for some $d_0$ and $\lambda_n=O(n^2)$ by induction. Here $\beta$ is the first Gegenbauer coefficient of $G$. With a weighted $\ell^{\infty}$ estimate of the Gegenbauer coefficients of $G$, the rigidity is established.

Nevertheless, the estimate of $a$ and the weighted $\ell^{\infty}$ estimates of the Gegenbauer coefficients are unnecessarily strong and hard to prove even when $N=8$. Furthermore, these estimates rely heavily on a priori numerical observations. Hence, this dimension-dependent strategy does not extend to general even $N$.

To establish an estimate for general even $N$, we resort to another key ingredient, namely a weighted $\ell^2$ estimate of $G$'s Gegenbauer coefficients proved in \cite{GLWY2026Beckner}, to show that $a\geq \frac{N-2}{N-1}$. This enables us to avoid the inductive proof of $a\leq \frac{d_0}{\lambda_n}$, which is one of the most difficult parts in the previous strategy. Hence, this argument avoids pointwise estimates of Gegenbauer polynomials required in the previous induction procedure. The rigidity will be established with the seminorm estimate and the corresponding auxiliary nonnegative quantity $D_N$.

For odd $N$, the operator $P_{r,N}$ is nonlocal. Thus, the local integration-by-parts computation of the seminorm is unavailable to estimate
\begin{equation*}
I_N=\int_{-1}^{1}(1-x^2)^{\frac{N-2}{2}}G^2P_{r,N}u dx.   
\end{equation*}
For $N=3$, $I_3$ can be represented by a trilinear form involving a difference quotient of $G$. For $N=5,7$, we establish the relationship between the corresponding cubic forms and $I_3$. Furthermore, instead of using second-order ODEs, we use an integral representation to derive pointwise estimates of $G'$ as well as $P_{r,N-2j}u$ for $0\leq j\leq \frac{N-3}{2}$. We believe that a similar strategy may work in all odd dimensions $N\ge 9$.

\subsection{Outline of the paper}

This paper is organized as follows. In Section \ref{preliminary}, we gather some properties of Gegenbauer polynomials, expand $G$ in terms of Gegenbauer polynomials, and cite some known results from \cite{GHW2022,GLWY2026Beckner}, including a weighted $\ell^2$ estimate of Gegenbauer coefficients of $G$. Section \ref{seminorm estimates even} presents a refined estimate of the seminorm $\lfloor G\rfloor^2$ by a carefully designed integration-by-parts formula. Section \ref{seminorm estimates odd} establishes the seminorm estimates for $N=3,5,7$ by descent to a three--dimensional nonlocal difference quotient identity. With these key estimates, we prove Theorem \ref{main} for general even $N$ and Theorem \ref{mainodd} for $N=3,5,7$ in Section \ref{ProofMain}. The integration-by-parts formula is proved in Appendix \ref{NewIBPproof}. Appendix \ref{RNkbound} establishes the coefficient bound of $R_{N,k}$. 

\section{Preliminaries}\label{preliminary}
In this section, we collect some properties of Gegenbauer polynomials and some known facts about the equation.

\subsection{Gegenbauer polynomials}

The Gegenbauer polynomial of order $\nu$ and degree $k$ (see \cite{Mori1998}) is given by
\begin{equation}\label{Rodrigues}
C_{k}^{\nu}(x)=\frac{(-1)^k}{2^k k!}\frac{\Gamma(\nu+\frac{1}{2})\Gamma(k+2\nu)}{\Gamma(2\nu)\Gamma(\nu+k+\frac{1}{2})}(1-x^2)^{-\nu+\frac{1}{2}}\frac{d^k}{dx^k} (1-x^2)^{k+\nu-\frac{1}{2}}.
\end{equation}

$C_{k}^{\nu}$ is an even function if $k$ is even and it is odd if $k$ is odd. The derivative of $C_{k}^{\nu}$ satisfies
\begin{equation}\label{201}
\frac{d}{dx}C_{k}^{\nu}(x)=2\nu C_{k-1}^{\nu+1}(x).
\end{equation}

Let $F_k^\nu$ be the normalization of $C_{k}^{\nu}$ such that $F_k^\nu(1)=1$, i.e.
\begin{equation*}%\label{Fknu}
F_k^\nu=\frac{k!\Gamma(2\nu)}{\Gamma(k+2\nu)} C_{k}^{\nu}.
\end{equation*}
Then $F_k^\nu$ satisfies the following differential equation
\begin{equation}\label{DE}
(1-x^2)(F_{k}^\nu)''-(2\nu+1)x(F_k^\nu)'+k(k+2\nu)F_k^\nu=0,
\end{equation}
and \eqref{201} becomes
\begin{align*}%\label{derivative}
(F_{k}^\nu)'=\frac{k(k+2\nu)}{2\nu+1}F_{k-1}^{\nu+1}.
\end{align*}

% It is also useful to introduce the following expressions using hypergeometric functions (see \cite{Olver2010} for instance)
% \begin{align*}%\label{hyper odd}
% F_{2m+1}^\nu(\cos \theta)= {_2}F_1(-m,m+\nu+1;\nu+\frac{1}{2};\sin^2\theta) \cos\theta,
% \end{align*}
% \begin{align*}%\label{hyper even}
% F_{2m}^{\nu}(\cos \theta)= {_2}F_1(-m,m+\nu;\nu+\frac{1}{2};\sin^2\theta),
% \end{align*}
% where we recall the hypergeometric function is defined for $|x|<1$ by power series
% \begin{equation*}
% {_2}F_1(a,b;c;x)=\sum_{k=0}^\infty \frac{(a)_k(b)_k}{(c)_k}\frac{x^k}{k!}.
% \end{equation*}

On $\mathbb{S}^N$, the corresponding Gegenbauer polynomial is $C_{k}^{\frac{N-1}{2}}$. To simplify notation, in what follows we will write $F_k$ for $F_k^\frac{N-1}{2}$, and there should be no danger of confusion.

The function $F_k$ satisfies the orthogonality condition
\begin{equation}\label{Fknorm}
\begin{aligned}
    \int_{-1}^{1}(1-x^2)^{\frac{N-2}{2}}F_{k}F_ldx&=\frac{2^{N-1}\Gamma(\frac{N}{2})^2 \Gamma(k+1)}{(k+N-2)!(2k+N-1)}\delta_{kl},
\end{aligned}    
\end{equation}
where $\lambda_k=k(k+N-1)$. Furthermore, it is an eigenfunction of the axial restriction of the Paneitz operator:
\begin{equation}\label{Paneitzeigen}
P_{r,N} F_k=
\begin{cases}
\frac{\Gamma(k+N)}{\Gamma(k)}F_k,&\qquad k\geq 1,\\
0,&\qquad k=0.
\end{cases}
\end{equation}

Another useful function is $F_k'$. For simplicity, in the rest of the paper, we normalize $\Tilde F_k'$ as 
\begin{equation*}
\Tilde{F}_{k}'=\frac{N}{\lambda_k}F_k'=\frac{k   !\Gamma(N+1)}{\Gamma(k+N-1)\lambda_k}C_{k-1}^{\frac{N+1}{2}}
\end{equation*}
so that $\Tilde{F}_{k}'(1)=1$.

$F_k$ and $\tilde F_k'$ also satisfy the following recurrence formulas
\begin{align}
xF_k=&\frac{k}{2k+N-1}F_{k-1}+\frac{k+N-1}{2k+N-1}F_{k+1}, \label{recurence1}\\
(1-x^2)F_k'=&\frac{k(k+N-1)}{2k+N-1}(F_{k-1}-F_{k+1}) \label{recurence2}\\
F_k=&\frac{(k+N-1)(k+N)}{N(2k+N-1)}\Tilde{F}_{k+1}'-\frac{k(k-1)}{N(2k+N-1)}\Tilde{F}_{k-1}' \label{recurence3}.
\end{align}

% Its generating function is
% \begin{equation}\label{generate}
% \sum_{k=0}^{\infty}\frac{r^k}{\int_{-1}^{1}(1-x^2)^{\frac{N-2}{2}}F_{k}^2 dx}F_k(z)=\frac{\Gamma (\frac{N+1}{2})}{\sqrt{\pi}\Gamma(\frac{N}{2})}\frac{1-r^2}{(1-2rz+r^2)^{\frac{N+1}{2}}}
% \end{equation}
% for $z\in (-1,1)$.

We recall the following formulas about Gegenbauer polynomials, which will be useful in Section \ref{weightedl^2est}. 
\begin{lemma}[\cite{Olver2010}, (18.17.5)]\label{prodform}
For any $k\geq 1$, we have
\begin{equation*}
    \tilde F_k'(x)\tilde F_k'(y)=\frac{\Gamma(\frac{N+2}{2})}{\sqrt{\pi}\Gamma(\frac{N+1}{2})}\int_{-1}^{1}\tilde F_k'\left(xy+t\sqrt{(1-x^2)(1-y^2)}\right)(1-t^2)^{\frac{N-1}{2}}dt.
\end{equation*}
\end{lemma}

\begin{lemma}[\cite{GLWY2026Beckner}, Lemma 2.2]
Consider
\begin{equation*}
    K_N(x,y):=\sum_{k=1}^{\infty}(2k+N-1) \tilde F_k'(x)\tilde F_k'(y).   
\end{equation*}
Then we have
\begin{equation}\label{KNxx}
    K_N(x,x)=\frac{N^2}{(N-1)(1-x^2)}
\end{equation}
for any $-1<x<1$, and
\begin{equation}\label{KNxy}
    (1-xy)K_N(x,y)\leq \frac{N^2}{N-1}
\end{equation}
for any $-1<x,y<1$.
\end{lemma}

\subsection{Integral identities}
We define the auxiliary function
\begin{equation*}%\label{def G}
G(x)=(1-x^2)u',
\end{equation*}
where $u$ is a solution to \eqref{axial} or \eqref{axialodd}. By elliptic regularity results, one can check that $u$ is a smooth function on $\mathbb{S}^N$. For a proof, we refer to \cite[Proposition 2.4]{GLWY2026Beckner}, and we remark that the argument actually doesn't require $\alpha<1$. In the following, we introduce some useful integral identities.

When $N$ is even, $G$ satisfies the equation
\begin{equation}\label{G}
\alpha(-1)^{\frac{N}{2}}[(1-x^2)^{\frac{N-2}{2}} G]^{(N-1)}+(N-1)!-\frac{(N-1)!\sqrt{\pi}\Gamma(\frac{N}{2})}{\Gamma(\frac{N+1}{2})\gamma}e^{Nu}=0,
\end{equation}
where
\begin{equation*}%\label{def gamma}
\gamma=\int_{-1}^{1}(1-x^2)^\frac{N-2}{2} e^{Nu}dx.
\end{equation*}
Differentiating the equation \eqref{G} and eliminating $e^{Nu}$ with \eqref{G}, we obtain the equation
\begin{equation}\label{Gequation}
\begin{aligned}
    (-1)^{\frac{N}{2}}(1-x^2)^{\frac{N}{2}}[(1-x^2)^{\frac{N-2}{2}}G]^{(N)}-\frac{N!}{\alpha}(1-x^2)^{\frac{N-2}{2}}G-(-1)^\frac{N}{2}N(1-x^2)^{\frac{N-2}{2}}G[(1-x^2)^{\frac{N-2}{2}}G]^{(N-1)}=0.
\end{aligned}
\end{equation}

When $N$ is odd, by Lemma \ref{PNwf} and \eqref{axialodd}, $G$ satisfies the equation
\begin{equation}\label{Godd}
P_{r,N}G=NG\left(P_{r,N}u+\frac{(N-1)!}{\alpha}\right)-NxP_{r,N}u.
\end{equation}

We expand $G$ in terms of Gegenbauer polynomials
\begin{equation}
\label{Gexpand}
G=a_0F_0+\beta x+a_2F_2(x)+\sum_{k=3}^{\infty}a_kF_k(x),
\end{equation}
and define
\begin{equation*}
%\label{gdef}
g:= (1-x^2)^{\frac{N-2}{2}} \frac{e^{Nu}}{\gamma}, \ a:=\int_{-1}^1 (1-x^2)gdx.
\end{equation*}

Then we have the following integral identities.
\begin{lemma}
Let $a$, $g$, and $G$ be as above. Then for every $N$, we have $a_0=0$ and the following integral identities
\begin{equation}\label{F1 beta}
    \int_{-1}^{1}(1-x^2)^{\frac{N-2}{2}}F_1Gdx=\frac{\sqrt{\pi}\Gamma(\frac{N}{2})}{2\Gamma(\frac{N+3}{2})}\beta,
\end{equation}
\begin{equation}\label{a}
    a=\int_{-1}^{1}(1-x^2)g=\frac{N}{N+1}(1-\alpha\beta),
\end{equation}
\begin{equation}\label{bbk}
    \int_{-1}^{1}(1-x^2)^{\frac{N-2}{2}}F_k Gdx=-\frac{2^{N-1}\Gamma(\frac{N}{2})^2 \Gamma(k)}{\alpha \Gamma(k+N)}\int_{-1}^{1}(1-x^2)gF_{k}'dx,\text{ }k\geq 2.
\end{equation}

If $N$ is even, then
\begin{equation}\label{D(N-2/2)G}
    \int_{-1}^{1}\left|[(1-x^2)^{\frac{N-2}{2}}G]^{(\frac{N-2}{2})}\right|^2dx=\frac{\sqrt{\pi}(N-2)!\Gamma(\frac{N}{2})}{\Gamma(\frac{N+3}{2})}\left(N+1-\frac{1}{\alpha}\right)\beta=\frac{2^N\Gamma(\frac{N}{2})^2}{(N+1)(N-1)}\left(N+1-\frac{1}{\alpha}\right)\beta.
\end{equation}

If $N$ is odd, then
\begin{equation}\label{D(N-2/2)Godd}
    \int_{-1}^{1}(1-x^2)^{\frac{N-2}{2}}G\mathcal{Q}_{N}G dx=\frac{\sqrt{\pi}(N-2)!\Gamma(\frac{N}{2})}{\Gamma(\frac{N+3}{2})}\left(N+1-\frac{1}{\alpha}\right)\beta=\frac{2^N\Gamma(\frac{N}{2})^2}{(N+1)(N-1)}\left(N+1-\frac{1}{\alpha}\right)\beta,
\end{equation}
where
\begin{equation*}
\mathcal{Q}_{N}:=
\begin{cases}
\left(-(1-x^2)\frac{d^2}{dx^2}+3x\frac{d}{dx}+1\right)^{\frac{1}{2}}, &N=3,\\
\left(-(1-x^2)\frac{d^2}{dx^2}+Nx\frac{d}{dx}+(\frac{N-1}{2})^2\right)^{\frac{1}{2}}\prod_{k=1}^{\frac{N-3}{2}}\left(-(1-x^2)\frac{d^2}{dx^2}+Nx\frac{d}{dx}+k(N-k-1)\right), &N\geq 5.
\end{cases}
\end{equation*}
Moreover, we have
\begin{equation}\label{xGQ_NG}
    \int_{-1}^{1}(1-x^2)^{\frac{N-2}{2}}xP_{r,N}u dx=\frac{N-1}{2}\int_{-1}^{1}(1-x^2)^{\frac{N-2}{2}}G\mathcal{Q}_{N}G dx
\end{equation}
\end{lemma}

\begin{proof}
The proof can be found in \cite[Lemma 2.5]{GLWY2026Beckner}.
\end{proof}

For odd $N$, the local integration-by-parts argument used in even dimensions is no longer applicable. Thus we need the formula about the nonlocal operator $P_{r,N}$, whose proof can be found in \cite[Lemma B.1]  {GLWY2026Beckner}.

\begin{lemma}[\cite{GLWY2026Beckner}, Lemma B.1]\label{PNwf}
For any odd $N\geq 3$ and smooth function $f$ on $[-1,1]$, we have
\begin{equation*}
P_{r,N}((1-x^2)f')=(1-x^2)(P_{r,N}f)'-NxP_{r,N}f.
\end{equation*}
\end{lemma}
Using this lemma and integration-by-parts, we obtain the following identity.
\begin{lemma}\label{descent}
For any odd $N\geq 3$ and smooth functions $f_1$, $f_2$, $f_3$, we have
\begin{equation*}
\int_{-1}^{1}(1-x^2)^{\frac{N}{2}-1}f_2f_3P_{r,N}((1-x^2)f_1')dx=-\int_{-1}^{1}(1-x^2)^{\frac{N}{2}}f_2'f_3P_{r,N}f_1dx-\int_{-1}^{1}(1-x^2)^{\frac{N}{2}}f_2f_3'P_{r,N}f_1dx.
\end{equation*}

\end{lemma}

% \begin{lemma}[Lemma 3.2 in \cite{GHW2022}]\label{G_j}
% For all $x\in (-1,1)$, we have
% \begin{equation}
    % G_j:=(-1)^j[(1-x^2)^j G]^{(2j+1)}\leq \frac{(2j+1)!}{\alpha},\ 0\leq j\leq \frac{N}{2}-1.
    % \end{equation}
% \end{lemma}

\subsection{A weighted $\ell^2$ estimate of Gegenbauer coefficients}\label{weightedl^2est}
In this subsection, we recall a weighted $\ell^2$ estimate for the Gegenbauer coefficients 
\begin{equation*}
b_{k}:=a_{k} \sqrt{ \int_{-1}^{1}(1-x^2)^{\frac{N-2}{2}}F_{k}^{2}dx},    
\end{equation*}
where $a_k$ is the $k$-th coefficient in the expansion of $G$, see \eqref{Gexpand}. 

The estimate for $b_k$ is one of the most important ingredients of this paper. In \cite{GHX2021}, Gui-Hu-Xie used \eqref{bbk} and the fact that
\begin{equation*}
|F_k'(x)|\leq |F_k'(1)|=\frac{\lambda_k}{N}
\end{equation*}
to estimate $b_k$ for $k\geq 2$ as follows
\begin{equation*}
\begin{aligned}
    b_{k}^{2}&=\frac{1}{\int_{-1}^{1}(1-x^2)^{\frac{N-2}{2}}F_{k}^{2}}\left[\frac{2^{N-1}\Gamma(\frac{N}{2})^2 \Gamma(k)}{\alpha \Gamma(k+N)}\int_{-1}^{1}(1-x^2)gF_{k}'\right]^2\\
    &\leq \frac{(k+N-2)!(2k+N-1)}{2^{N-1}\Gamma(\frac{N}{2})^2 k!}\left[\frac{2^{N-1}\Gamma(\frac{N}{2})^2 \Gamma(k)}{\alpha \Gamma(k+N)}\frac{\lambda_k}{N}a\right]^2\\
    &=\frac{2^{N-1}\Gamma(\frac{N}{2})^2 (2k+N-1)\Gamma(k+1)}{\alpha^2N^2\Gamma(k+N-1)}a^2.
\end{aligned}    
\end{equation*}
However, their estimates are not strong enough to reach the sharp constant $\alpha=\frac{1}{2}$ for $N=4,6,8$.

The works of Li-Wei-Ye \cite{LWY2022} and Gui-Li-Wei-Ye \cite{GLWY2025} refined this estimate and closed the gap $\alpha=\frac{1}{2}$ for $N=4,6,8$. However, the refined estimates are still difficult to generalize to arbitrary $N$.

We now derive a weighted $\ell^2$ estimate for $b_k$. To this end, we define
\begin{equation*}%\label{def Ak}
A_{k}:=\int_{-1}^{1}(1-x^2)\tilde{F}_{k}'g dx.
\end{equation*}
Recalling the definition of $g$, \eqref{AxialLrN} and \eqref{a}, we have
\begin{equation*}
\int_{-1}^{1}g dx=1,\ \int_{-1}^{1}xg dx=0\text{ and }\int_{-1}^{1}(1-x^2)g dx=a=\frac{N}{N+1}(1-\alpha\beta).
\end{equation*}

By \eqref{bbk}, we have
\begin{equation*}
    \frac{\Gamma(k+N-1)}{\Gamma(k+1)}b_k^2=\frac{2^{N-1}\Gamma(\frac{N}{2})^2}{\alpha^2 N^2}(2k+N-1)A_k^2
\end{equation*}
for $k\geq 2$.

With the above properties of $g$, the next theorem provides a weighted $\ell^2$ estimate on $A_k$ and hence on $b_k$, which yields a lower bound of $a$.
\begin{theorem}[\cite{GLWY2026Beckner}, Theorem 3.1]\label{bk}
Let $N\geq 2$ be an integer and let $a$, $A_k$ be as above. Then we have
\begin{equation*}
    Na-(N+1)a^2\leq \sum_{k=2}^{\infty}(2k+N-1)A_k^2\leq \frac{N^2}{N-1}\frac{a}{2-a}-(N+1)a^2.
\end{equation*}
As a consequence, we have $a\geq \frac{N-2}{N-1}$.
\end{theorem}

\section{Seminorm estimates}\label{seminorm estimates}
This section is devoted to refined estimates on the seminorm $\lfloor G\rfloor^2$. These estimates are one of the key ingredients of this paper. The seminorm $\lfloor G\rfloor^2$ estimate is obtained by a new integration by parts strategy to deal with the cubic term $I_N$.

\subsection{Seminorm estimates for $N$ even}\label{seminorm estimates even}

In this subsection, we assume $N$ is even. To get an estimate on $\beta$ and $a=\frac{N}{N+1}(1-\alpha\beta)$, we need an estimate of the following seminorm $\lfloor G\rfloor^2$ defined by
\begin{equation}\label{Gfloor def}
\lfloor G\rfloor^2:=(-1)^\frac{N}{2}\int_{-1}^{1}(1-x^2)^{\frac{N-2}{2}}[(1-x^2)^{\frac{N}{2}}G']^{(N-1)}G dx.
\end{equation}

Testing \eqref{Gequation} against $G$ and integrating by parts, we have
\begin{equation}\label{Gfloor I}
\lfloor G\rfloor^2
=-\frac{N(N-1)}{2}\int_{-1}^{1}|[(1-x^2)^\frac{N-2}{2} G]^{(\frac{N-2}{2})}|^2 dx
+\frac{N!}{\alpha}\int_{-1}^{1}(1-x^2)^{\frac{N}{2}-1}G^2 dx+ NI_N,
\end{equation}
where
\begin{equation*}
I_N=(-1)^{\frac{N}{2}}\int_{-1}^{1}(1-x^2)^{\frac{N-2}{2}}G^2[(1-x^2)^{\frac{N-2}{2}}G]^{(N-1)} dx.
\end{equation*}

The main difficulty in the even-dimensional case is the cubic term $I_N$. The next proposition provides an integration-by-parts formula that rewrites $I_N$ in a controllable form.

\begin{proposition}\label{newIBP}
Let $N\geq 4$ be an even integer. For $G=(1-x^2)u'$, we define $\hat{G}_j$ recursively by $\hat{G}_0:=G'$ and 
\begin{equation*}
\hat{G}_j:=-(1-x^2)\hat{G}_{j-1}''+(N+2)x\hat{G}_{j-1}'+(j+1)(N-j)\hat{G}_{j-1}   
\end{equation*}
for $1\leq j\leq \frac{N}{2}-1$. 
Then 
\begin{equation}\label{eq hat GM}
\hat G_{\frac{N}{2}-1}=(-1)^{\frac{N}{2}-1}[(1-x^2)^{\frac{N-2}{2}} G]^{(N-1)},
\end{equation}
and	we have
\begin{equation}\label{IN IBP}
    I_N=\sum_{r=1}^{M}\sum_{d=0}^{M-r}C_{r,d}\int_{-1}^{1}(1-x^2)^{\frac{N}{2}-1+r}(G^{(r)})^2\hat G_{M-r-d} dx,
\end{equation}
where $M=\frac{N}{2}-1$ and
\begin{equation*}
    C_{r,d}=\binom{M}{r}\binom{M-r}{d}(r)_d(M+r)_d\frac{3(M+r)+d}{3(M+r)}.
\end{equation*}
Here
\begin{equation*}
(t)_k:=
\begin{cases}
    t(t+1)\cdots (t+k-1),&\qquad k\geq 1,\\
    1,&\qquad k=0
\end{cases}    
\end{equation*}
is the Pochhammer symbol for $t\in \mathbb{R}$ and $k\in \mathbb{N}$. 
\end{proposition}

The proof of Proposition \ref{newIBP} is lengthy and technical, so we defer it to Appendix \ref{NewIBPproof}. 

With Proposition \ref{newIBP}, the following proposition establishes pointwise estimates of $\hat{G}_j$. In this way, $I_N$ and the seminorm $\lfloor G\rfloor^2$ can be estimated. 
\begin{proposition}\label{Gfloor est}
For every even $N\geq 4$, we have
\begin{equation*}
    \hat{G}_j\leq \frac{(j+1)!(N-1)!}{(N-j-1)!\alpha}
\end{equation*}
for any $0\leq j\leq \frac{N}{2}-1$. Consequently, we have
\begin{equation*}
    I_N\leq \frac{(N-1)!}{\alpha}\sum_{r=1}^{M}\sum_{d=0}^{M-r}\frac{(M-r-d+1)!}{(M+r+d+1)!}C_{r,d}\sum_{k=1}^{\infty}\left(\prod_{i=0}^{r-1}(\lambda_k-\lambda_i )\right) b_k^2
\end{equation*}
and hence
\begin{equation*}
    \lfloor G\rfloor^2\leq -\frac{N(N-1)}{2}\int_{-1}^{1}\Big|[(1-x^2)^\frac{N-2}{2} G]^{(\frac{N-2}{2})}\Big|^2 dx+\frac{1}{\alpha}\sum_{k=1}^{\infty}R_{N,k}b_k^2,
\end{equation*}
where $M=\frac{N}{2}-1$ and
\begin{equation}\label{RNk}
    R_{N,k}=N!+N!\sum_{r=1}^{M}\sum_{d=0}^{M-r}\frac{(M-r-d+1)!}{(M+r+d+1)!}C_{r,d}\prod_{i=0}^{r-1}(\lambda_k-\lambda_i ).
\end{equation}

\end{proposition}
\begin{proof}

Set $C_j:=\frac{(j+1)!(N-1)!}{(N-j-1)!}$. Then we have $C_{j+1}=(j+2)(N-j-1)C_{j}$.

First, by \eqref{G} and \eqref{eq hat GM}, we have
\begin{equation*}
    \alpha \hat G_{\frac{N}{2}-1}=(N-1)!-\frac{(N-1)!\sqrt{\pi}\Gamma(\frac{N}{2})}{\Gamma(\frac{N+1}{2})\gamma}e^{Nu}\leq (N-1)!=C_{\frac{N}{2}-1}.
\end{equation*}

Now suppose that for some $0\leq j\leq \frac{N}{2}-2$, $\hat{G}_{j+1}\leq \frac{C_{j+1}}{\alpha}$. Then we have
\begin{equation*}
    (1-x^2)\hat{G}_{j}''-(N+2)x\hat{G}_{j}'-(j+2)(N-j-1)\hat{G}_{j}=-\hat{G}_{j+1}\geq -\frac{C_{j+1}}{\alpha}.
\end{equation*}

We claim that
\begin{equation*}
    \hat{G}_{j}\leq \frac{C_j}{\alpha}.
\end{equation*}

To prove the claim, we denote $M_j:=\max\limits_{-1\leq x\leq 1}\hat{G}_{j}(x)$ and consider separately an interior maximum and an endpoint maximum.

{\em Case 1:} If $M_j$ is attained at some point $x_0\in (-1,1)$, then
\begin{equation*}
    \hat{G}_{j}'(x_0)=0,\ \hat{G}_{j}''(x_0)\leq 0
\end{equation*}
and the desired estimate follows. 

{\em Case 2:} If $M_j$ is attained at $1$ or $-1$, without loss of generality, suppose there exists a sequence $x_n\to 1^-$ such that
\begin{equation*}
    M_j=\lim\limits_{n\to\infty}\hat{G}_{j}(x_n).
\end{equation*}		
In this case, we let $r=\sqrt{1-x^2}$ and write
\begin{equation*}
    G(x)=\Bar{G}(r),\ \hat{G}_{j}(x)=H_j(r)\text{ and }u(x)=\Bar{u}(r)\text{ for }r\in [0,1),\ x\in (0,1].
\end{equation*}
Then we can extend $\Bar{u}(r)$ to be a smooth even function on $(-\frac{1}{2},\frac{1}{2})$. 

Direct calculation yields that $H_j$
is an even function with respect to $r$ and can be extended to be a smooth function near $r=0$. Now we can write
\begin{align*}
    H_j(r)
    &=c_1+c_2 r^2+c_3 r^4+O(r^6),\\
    xH_j'(x)&=-2c_2+O(r^2),\\
    (1-x^2)H_j''(x)&=(-2c_2+8c_3)r^2+O(r^4)
\end{align*}
near $r=0$. Since $H_j(r)$ attains its local maximum at $r=0$, we have $c_2\leq 0$ and hence
\begin{equation*}
    \lim\limits_{x\to 1}xH_j'(x)\geq 0,\ \lim\limits_{x\to 1}(1-x^2)H_j''(x)=0.
\end{equation*}

Then we obtain $M_j\leq \frac{C_j}{\alpha}$ and hence the claim is proved.

Substituting the estimates of $\hat{G}_{j}$ into Proposition \ref{newIBP} and combining with \eqref{Gfloor I}, we get the desired estimates for $I_N$ and $\lfloor G\rfloor^2$.
\end{proof}

\subsection{Seminorm estimates for $N=3,5,7$}\label{seminorm estimates odd}

In this subsection, we will establish the estimates for seminorm
\begin{equation*}
\lfloor G\rfloor^2:=\int_{-1}^{1} (1-x^2)^{\frac{N}{2}-1}GP_{r,N} G dx
\end{equation*}
for $N=3,5,7$.

Testing Lemma \ref{PNwf} against $G$ and using \eqref{axialodd} and \eqref{xGQ_NG}, we get
\begin{equation*}
\lfloor G\rfloor^2=-\frac{N(N-1)}{2}\int_{-1}^{1} (1-x^2)^{\frac{N}{2}-1}G\mathcal{Q}_{N} G dx+\frac{N!}{\alpha}\int_{-1}^{1}(1-x^2)^{\frac{N}{2}-1}G^2+N I_N,
\end{equation*}
where
\begin{equation*}
I_N:=\int_{-1}^{1}(1-x^2)^{\frac{N}{2}-1}G^2P_{r,N}u dx.
\end{equation*}

There is no direct analogue of Proposition \ref{Gfloor est} when $N$ is odd. Instead, the following proposition gives an estimate for $I_N$ for $N=3,5,7$. 
\begin{proposition}\label{seminormodd}
For $N=3,5,7$, we have
\begin{equation*}
I_N\leq \frac{\theta_N}{\alpha}\int_{-1}^{1}(1-x^2)^{\frac{N}{2}-1}G\mathcal{Q}_{N}Gdx-\frac{\theta_N(N-2)!}{\alpha}\int_{-1}^{1}(1-x^2)^{\frac{N}{2}-1}G^2dx,
\end{equation*}
where $\theta_3=\frac{2}{3}$, $\theta_5=\frac{136}{75}$, $\theta_7=\frac{12}{5}$.
Consequently, we have
\begin{equation*}
\lfloor G\rfloor^2\leq \left(\frac{N\theta_N}{\alpha}-\frac{N(N-1)}{2}\right)\int_{-1}^{1}(1-x^2)^{\frac{N}{2}-1}G\mathcal{Q}_{N}Gdx+\frac{N!-N\theta_N(N-2)!}{\alpha}\int_{-1}^{1}(1-x^2)^{\frac{N}{2}-1}G^2dx.
\end{equation*}
\end{proposition}

The remaining part of this subsection is devoted to proving it by descending the nonlocal operator to a three-dimensional nonlocal difference quotient identity. We first establish pointwise bounds for $G'$ and $P_{r,N-2j}u$, where we recall the operator $P_{r,N-2j}$ is defined in a similar way to \eqref{def P rN}.

\begin{lemma}\label{GPbound}
Let $N\geq 3$ be an odd integer. Then we have 
\begin{equation*}
    \sup_{[-1,1]} G'\leq
\begin{cases}
\frac{1}{\alpha},&\qquad N=3,\\
\frac{5N-8}{N(N-1)\alpha},&\qquad N\geq 5.
\end{cases}
\end{equation*}

Furthermore, for any $0 \leq j\leq \frac{N-3}{2}$, we have
\begin{equation*}
P_{r,N-2j}u\geq -\frac{(N-2j-1)!}{\alpha}.
\end{equation*}

% \begin{equation*}
%     \hat{G}_j\leq \frac{(j+1)!(N-1)!}{(N-j-1)!\alpha}
% \end{equation*}
% for any $0\leq j\leq \frac{N-3}{2}$. 
\end{lemma}
\begin{proof}
We first prove the pointwise estimate for $G'$. For simplicity, we define
\begin{equation*}
\mathcal{A}:=\left(-(1-x^2)\frac{d^2}{dx^2}+(N+2)x\frac{d}{dx}+(\frac{N+1}{2})^2\right)^{\frac{1}{2}}-\frac{N+1}{2}.
\end{equation*}
Then $\mathcal{A}\tilde F_k'=(k-1)\tilde F_k'$.

We first claim that

\textbf{Claim 1:}
\begin{equation*}
\mathcal{L_N}G':=\prod_{l=2}^{N-1}(\mathcal{A}+l)G'=-P_{r,N}u.
\end{equation*}
Since $\{F_k\}$ forms a complete orthogonal basis in $L^2((-1,1);(1-x^2)^{\frac{N}{2}-1}dx)$, it suffices to prove the identity for $u=F_k$.

By \eqref{recurence3}, it follows that
\begin{equation*}
\begin{aligned}
\prod_{l=2}^{N-1}(\mathcal{A}+l)G'
=& \frac{\Gamma(k+N)}{\Gamma(k)}\left(-\frac{(k+N-1)(k+N)}{N(2k+N-1)}\tilde F_{k+1}'+\frac{k(k-1)}{N(2k+N-1)}\tilde F_{k-1}'\right)\\
=&-\frac{\Gamma(k+N)}{\Gamma(k)}F_k.
\end{aligned}
\end{equation*}

For $u=F_k$, the right-hand side equals
\begin{equation}
-P_{r,N}u=-\frac{\Gamma(k+N)}{\Gamma(k)}F_k
\end{equation}
by \eqref{Paneitzeigen}. Claim 1 follows.

Since
\begin{equation*}
\prod_{l=2}^{N-1}(\mathcal{A}+l)1=(N-1)!,     
\end{equation*}
by \eqref{axialodd}, we have
\begin{equation}\label{1alpha-G'}
\mathcal{L_N}\left(\frac{1}{\alpha}-G'\right)=\frac{(N-1)!\sqrt{\pi}\Gamma(\frac{N}{2})}{\alpha\Gamma(\frac{N+1}{2})\gamma}e^{Nu}.
\end{equation}

Define
\begin{equation*}
\mathcal K_N(x,y):=\sum_{k=1}^{\infty}\frac{\Gamma(k+1)}{\Gamma(k+N-1)}\frac{\tilde F_k'(x)\tilde F_k'(y)}{\int_{-1}^{1}(1-z^2)^{\frac{N}{2}}(\tilde F_k'(z))^2dz}
\end{equation*}

\textbf{Claim 2:} For any smooth function $f$, we have
\begin{equation*}
\mathcal{L}_N^{-1}f(x)=\int_{-1}^{1}(1-y^2)^{\frac{N}{2}}\mathcal K_N(x,y)f(y)dy.
\end{equation*}
Indeed, recall that $\tilde F_k'=F_{k-1}^{\frac{N+1}{2}}$ . Hence $\{\tilde F_k'\}_{k\ge1}$ forms a complete orthogonal
basis of $L^2\left((-1,1);(1-x^2)^{\frac{N}{2}}\,dx\right)$.

\begin{equation*}
\mathcal{L}_N \tilde F_k'=\prod_{l=2}^{N-1}(k+l-1)\tilde F_k'=\frac{\Gamma(k+N-1)}{\Gamma(k+1)}\tilde F_k'.
\end{equation*}

Thus, for any smooth $f$, we have
\begin{equation*}
f=\sum_{k=1}^{\infty}f_k\tilde F_k',   
\end{equation*}
with
\begin{equation*}
f_k=\frac{\int_{-1}^{1}(1-y^2)^{\frac{N}{2}}f(y)\tilde F_k'dy}{\int_{-1}^{1}(1-y^2)^{\frac{N}{2}}(\tilde F_k')^2dy}.    
\end{equation*}
Then we have
\begin{equation*}
\mathcal{L}_N^{-1}f=\sum_{k=1}^{\infty}
\frac{\Gamma(k+1)}{\Gamma(k+N-1)}f_k\tilde F_k',    
\end{equation*}
which gives Claim 2.

Applying Claim 2 to \eqref{1alpha-G'}, we have
\begin{equation*}
\frac{1}{\alpha}-G'=\frac{(N-1)!\sqrt{\pi}\Gamma(\frac{N}{2})}{\alpha\Gamma(\frac{N+1}{2})\gamma}\int_{-1}^{1}(1-y^2)^{\frac{N}{2}}\mathcal{K}_N(x,y)e^{Nu(y)}dy=\frac{(N-1)!\sqrt{\pi}\Gamma(\frac{N}{2})}{\alpha\Gamma(\frac{N+1}{2})}\int_{-1}^{1}(1-y^2)\mathcal{K}_N(x,y)g(y)dy.
\end{equation*}

Furthermore, by Lemma \ref{prodform} and the identity
\begin{equation*}
\frac{\Gamma(k+1)}{\Gamma(k+N-1)}=\frac{1}{(N-3)!}\int_0^1 r^{k}(1-r)^{N-3}dr,\qquad k\geq 1,
\end{equation*}
we can write $\mathcal K_N$ as 
\begin{equation*}
\mathcal K_N(x,y)=\frac{N+1}{2\pi (N-3)!}\int_{0}^{1}r(1-r)^{N-3}(1-r^2)\int_{-1}^{1}\frac{(1-t^2)^{\frac{N-1}{2}}}{\left[1-2r(xy+t\sqrt{(1-x^2)(1-y^2)})+r^2\right]^{\frac{N+3}{2}}}dtdr\geq 0.   
\end{equation*}

Since the function $s\mapsto (1-2rs+r^2)^{-\frac{N+3}{2}}$ is convex and 
\begin{equation*}
\int_{-1}^{1}\frac{\Gamma(\frac{N+2}{2})}{\sqrt{\pi}\Gamma(\frac{N+1}{2})}(1-t^2)^{\frac{N-1}{2}}dt=1,
\end{equation*}
by Jensen's inequality, we have
\begin{equation*}
\begin{aligned}
\mathcal K_N(x,y)
\geq& \frac{N+1}{2\pi (N-3)!}\frac{\sqrt{\pi}\Gamma(\frac{N+1}{2})}{\Gamma(\frac{N+2}{2})}\int_0^1\frac{r(1-r)^{N-3}(1-r^2)}{(1-2rxy+r^2)^{\frac{N+3}{2}}}dr\\
=& \frac{\Gamma(\frac{N+1}{2})}{2\sqrt{\pi} (N-3)!(N-1)\Gamma(\frac{N+2}{2})}\frac{1}{(1-xy)^2}.
\end{aligned}
\end{equation*}
It follows that
\begin{equation*}
\begin{aligned}
\frac{1}{\alpha}-G'(x)
=&\frac{(N-1)!\sqrt{\pi}\Gamma(\frac{N}{2})}{\alpha\Gamma(\frac{N+1}{2})}\int_{-1}^{1}(1-y^2)\mathcal{K}_N(x,y)g(y)dy\\
\geq &\frac{N-2}{N\alpha}\int_{-1}^{1}\frac{1-y^2}{(1-xy)^2}g(y)dy.
\end{aligned}
\end{equation*}

For $N=3$, it follows that $G'(x)\leq \frac{1}{\alpha}$ for all $x$. 

For $N\geq 5$, since 
\begin{equation*}
\frac{1-y^2}{(1-xy)^2}\geq 1+2xy-3y^2
\end{equation*}
and
\begin{equation*}
\int_{-1}^{1}gdy=1,\qquad \int_{-1}^{1}ygdy=0,\qquad \int_{-1}^{1}y^2gdy=1-a,
\end{equation*}
using $a\geq \frac{N-2}{N-1}$ from Theorem \ref{bk}, we have
\begin{equation*}
\frac{1}{\alpha}-G'\geq \frac{N-2}{N\alpha}\int_{-1}^{1}\frac{1-y^2}{(1-xy)^2}g(y)dy\geq \frac{N-2}{N\alpha}(1-3(1-a))\geq \frac{(N-2)(N-4)}{N(N-1)\alpha}.
\end{equation*}

Equivalently, 
\begin{equation*}
    G'(x)\leq \frac{5N-8}{N(N-1)\alpha}
\end{equation*}
as desired.

Next, we estimate $P_{r,N-2j}u$ for every $0\leq j\leq \frac{N-3}{2}$.

The estimate of $P_{r,N}$ follows from \eqref{axialodd}. Since for any $1\leq j\leq \frac{N-3}{2}$, 
\begin{equation*}
-(1-x^2)\frac{d^2}{dx^2}P_{r,N-2j}u+(2N-4j+2)x\frac{d}{dx}P_{r,N-2j}u+(N-2j)(N-2j+1)P_{r,N-2j}u=P_{r,N-2j+2}u,
\end{equation*}
the estimate of $P_{r,N-2j}u$ follows from a similar argument as in the proof of Proposition \ref{Gfloor est}.

\end{proof}

Next, we derive two important identities when $N=3$. For an arbitrary smooth function $H$, we define
\begin{equation*}
\delta H(x,y)=
\begin{cases}
\frac{H(x)-H(y)}{x-y},&\qquad x\neq y,\\
H'(x),&\qquad x=y.
\end{cases}
\end{equation*}

The following identity identifies the nonlocal quadratic form of $\delta H$ with the energy involving $\mathcal{Q}_{3}$. It serves as an integration-by-parts formula in three dimensions.
\begin{lemma}\label{deltaH2}
For any smooth function $H$, we have
\begin{equation*}
\int_{-1}^{1}\int_{-1}^{1}(1-x^2)^{\frac12}(1-y^2)^{\frac12}(\delta H)^2dxdy=2\pi \int_{-1}^{1}(1-x^2)^{\frac12}H(\mathcal{Q}_{3}H-H)dx.
\end{equation*}
\end{lemma}
\begin{proof}
In fact, we will prove the following polarized version:
\begin{equation}\label{deltaH1H2}
\int_{-1}^{1}\int_{-1}^{1}(1-x^2)^{\frac12}(1-y^2)^{\frac12}\delta H_1 \delta H_2dxdy=2\pi \int_{-1}^{1}(1-x^2)^{\frac12}H_1(\mathcal{Q}_{3}H_2-H_2)dx.
\end{equation}
Since both sides are bilinear forms in $(H_1,H_2)$, it suffices to prove it for $H_1=F_k^1$ and $H_2=F_l^1$. For $m\geq 0$, we define
\begin{equation*}
S_m(x,y)=\sum_{i+j=m}(i+1)(j+1)F_i^1(x)F_j^1(y)
\end{equation*}
and $S_m=0$ for $m<0$. By \eqref{recurence1}, we have
\begin{equation*}
\delta F_k^1=\frac{2}{k+1}S_{k-1}.    
\end{equation*}

By the orthogonality, we have
\begin{equation*}
\int_{-1}^{1}\int_{-1}^{1}(1-x^2)^{\frac12}(1-y^2)^{\frac12}S_{k-1}S_{l-1}dxdy=\frac{\pi^2}{4}k\delta_{kl}    
\end{equation*}
and hence
\begin{equation*}
\int_{-1}^{1}\int_{-1}^{1}(1-x^2)^{\frac12}(1-y^2)^{\frac12}\delta F_k^1\delta F_l^1 dxdy=\frac{\pi^2k}{(k+1)^2}\delta_{kl}    
\end{equation*}

On the other hand, we have
\begin{equation*}
2\pi \int_{-1}^{1}(1-x^2)^{\frac12}F_k^1(\mathcal{Q}_{3}F_l^1-F_l^1)dx=\frac{\pi^2k}{(k+1)^2}\delta_{kl}.
\end{equation*}
Hence, \eqref{deltaH1H2} follows and we complete the proof.
\end{proof}

Define the trilinear form
\begin{equation*}
\mathcal{D}(H_1,H_2,H_3)=\int_{-1}^{1}\int_{-1}^{1}(1-x^2)^{\frac12}(1-y^2)^{\frac12}\delta H_1 \delta H_2\delta H_3dxdy
\end{equation*}

The following lemma is a trilinear counterpart of Lemma \ref{deltaH2}.
\begin{lemma}
For smooth $u_i$ and $G_i=(1-x^2)u_i'$, $i=1,2,3$, we have
\begin{equation}\label{DG1G2G3}
\mathcal{D}(G_1,G_2,G_3)=\pi \int_{-1}^{1}\sqrt{1-x^2}[P_{r,3}u_1 G_2G_3+P_{r,3}u_2 G_1G_3+P_{r,3}u_3 G_1G_2] dx.
\end{equation}
\end{lemma}
\begin{proof}
Since both sides are trilinear forms in $(u_1,u_2,u_3)$, it suffices to prove the lemma for $u_1=F_i^1$, $u_2=F_j^1$ and $u_3=F_k^1$. 

Note that $F_k^1=\frac{1}{k+1}U_k$, where $U_k$ is the Chebyshev polynomial of the second kind. Then we have
\begin{equation*}
U_iU_j=\sum_{r=0}^{\min\{i,j\}}U_{i+j-2r},
\end{equation*}
i.e.
\begin{equation*}
F_i^1F_j^1=\sum_{r=0}^{\min\{i,j\}}\frac{i+j-2r+1}{(i+1)(j+1)}F_{i+j-2r}^1,
\end{equation*}

Hence, we have
\begin{equation}\label{FiFjFk}
\int_{-1}^{1}\sqrt{1-x^2}F_i^1F_j^1F_k^1 dx=\frac{\pi}{2(i+1)(j+1)(k+1)}\chi(i,j,k),
\end{equation}
where $\chi(i,j,k)=1$ when $|i-j|\leq k\leq i+j$ and $i+j+k$ is even and $\chi(i,j,k)=0$ otherwise.

Recall $S_m$ in the proof of Lemma \ref{deltaH2}. Then we have
\begin{equation*}
\begin{aligned}
\int_{-1}^{1}\int_{-1}^{1}\sqrt{1-x^2}\sqrt{1-y^2}S_pS_qS_r dxdy
&=\frac{\pi^2}{4}\left(\frac{p+q-r}{2}+1\right)\left(\frac{p+r-q}{2}+1\right)\left(\frac{q+r-p}{2}+1\right)\chi(p,q,r)\\
&=: \frac{\pi^2}{4}\Psi(p,q,r).
\end{aligned}
\end{equation*}
Here we adopt the convention that $S_{m}=0$ and $\Psi(p,q,r)=0$ if any one of the index is negative. 

By the recurrence formulas, we have
\begin{equation}\label{recurrence1S31}
(1-x^2)(F_k^1)'=\frac{k(k+2)}{2(k+1)}(F_{k-1}^1-F_{k+1}^1)
\end{equation}
and
\begin{equation*}
\delta G_k=\frac{k+2}{k+1}S_{k-2}-\frac{k}{k+1}S_k.
\end{equation*}

Substituting them into the left-hand side of \eqref{DG1G2G3}, we get 
\begin{equation}\label{DGiGjGk1}
\begin{aligned}
\mathcal{D}(G_i,G_j,G_k)
=&\frac{\pi^2}{4(i+1)(j+1)(k+1)}\left[(i+2)(j+2)(k+2)\Psi(i-2,j-2,k)-ijk\Psi(i,j,k)\right.\\
-&\left.\sum_{cyc}i(j+2)(k+2)\Psi(i,j-2,k-2)+\sum_{cyc}ij(k+2)\Psi(i,j,k-2)\right].
\end{aligned}
\end{equation}

On the other hand, since
\begin{equation*}
    P_{r,3}F_i^1=i(i+1)(i+2)F_i^1,
\end{equation*}
by \eqref{FiFjFk} and \eqref{recurrence1S31}, we have
\begin{equation}\label{P3FiGjGk}
\begin{aligned}
&\pi \sum_{cyc}\int_{-1}^{1}\sqrt{1-x^2}P_{r,3}F_i^1 G_jG_k dx\\
=&\frac{\pi^2}{8}\sum_{cyc}\frac{i(i+2)}{(j+1)(k+1)}\left[(j+2)(k+2)\chi(i,j-1,k-1)-(j+2)k\chi(i,j-1,k+1)\right.\\
-&\left.j(k+2)\chi(i,j+1,k-1)+jk\chi(i,j+1,k+1)]\right.    
\end{aligned}
\end{equation}

Let
\begin{equation*}
A=\frac{j+k-i}{2},\qquad B=\frac{i+k-j}{2},\qquad C=\frac{i+j-k}{2},\qquad
\end{equation*}

Comparing the cyclic summations \eqref{DGiGjGk1} and \eqref{P3FiGjGk} in 
Case 1: $A,B,C\geq 1$, Case 2: $A=0$, Case 3: $A=-1$, Case 4: $A\leq -2$, we see that
\begin{equation*}
\mathcal{D}(G_i,G_j,G_k)=    \pi \sum_{cyc}\int_{-1}^{1}\sqrt{1-x^2}P_{r,3}F_i^1 G_jG_k dx,
\end{equation*}
which is the desired result. We omit the details here.
\end{proof}

Now we give the proof of Proposition \ref{seminormodd} when $N=3$.
\begin{proof}[Proof of Proposition \ref{seminormodd} when $N=3$.]

By Lemma \ref{GPbound}, we have $G'\leq \frac{1}{\alpha}$. Hence by the mean-value theorem, $\delta G\leq \frac{1}{\alpha}$.

Then by \eqref{DG1G2G3} and Lemma \ref{deltaH2}, we have
\begin{equation*}
\begin{aligned}
I_3
=&\frac{1}{3\pi} \mathcal{D}(G,G,G)=\frac{1}{3\pi}\int_{-1}^{1}\int_{-1}^{1}(1-x^2)^{\frac12}(1-y^2)^{\frac12}(\delta G)^3 dxdy\\
\leq &\frac{1}{3\pi\alpha}\int_{-1}^{1}\int_{-1}^{1}(1-x^2)^{\frac12}(1-y^2)^{\frac12}(\delta G)^2 dxdy\\
= &\frac{2}{3\alpha}\int_{-1}^{1}(1-x^2)^{\frac12}G(\mathcal{Q}_{3}G-G) dx
\end{aligned}
\end{equation*}
as desired.
\end{proof}

To prove Proposition \ref{seminormodd} when $N=5$, we need the following spectral inequality, which shows that the $\mathcal{Q}_{5}$ energy dominates the $\mathcal{Q}_{3}$ energy.
\begin{lemma}\label{Q5Q3}
For any smooth function $G$ and $F=(1-x^2)G'$, we have
\begin{equation*}
\begin{aligned}
\frac{3}{5}\int_{-1}^{1}(1-x^2)^{\frac{3}{2}}G(\mathcal{Q}_{5}-6)Gdx 
\geq& \frac{27}{16}\int_{-1}^{1}(1-x^2)^{\frac{1}{2}}G(\mathcal{Q}_{3}-1)Gdx +\frac{9}{16}\int_{-1}^{1}(1-x^2)^{\frac{1}{2}}F(\mathcal{Q}_{3}-1)Fdx\\    
&+\frac{3}{4}\int_{-1}^{1}(1-x^2)^{\frac{1}{2}}F^2dx.
\end{aligned}
\end{equation*}
\end{lemma}
\begin{proof}
To avoid ambiguity, we will keep the superscripts of the Gegenbauer polynomials $F_k^1$ and $F_k^2$ in this proof.

We first expand $G=\sum_{k=0}^{\infty}c_kF_k^1$. By the recurrence formulas \eqref{recurence1} and \eqref{recurence2}, we have
\begin{equation*}
F_k^1=\frac{(k+2)(k+3)}{6(k+1)}F_k^2-\frac{k(k-1)}{6(k+1)}F_{k-2}^2
\end{equation*}
and
\begin{equation*}
(1-x^2)(F_k^1)'=\frac{k(k+2)}{2(k+1)}(F_{k-1}^1-F_{k+1}^1).
\end{equation*}
It follows that
\begin{equation*}
G=\sum_{k=0}^{\infty}\frac{k+2}{6}\left(\frac{k+3}{k+1}c_k-\frac{k+1}{k+3}c_{k+2}\right)F_k^2.
\end{equation*}
and
\begin{equation*}
F:=\sum_{k=0}^{\infty}f_kF_k^1=\sum_{k=0}^{\infty}\left(\frac{(k+1)(k+3)}{2(k+2)}c_{k+1}-\frac{(k-1)(k+1)}{2k}c_{k-1}\right)F_k^1,
\end{equation*}
where the second term is omitted when $k=0$.

Since
\begin{equation*}
\mathcal{Q}_{5} F_k^2=(k+1)(k+2)(k+3)F_k^2,
\end{equation*}
by \eqref{Fknorm}, we have
\begin{equation*}
\int_{-1}^{1}(1-x^2)^{\frac{3}{2}}G(\mathcal{Q}_{5}-6)Gdx=\frac{\pi}{8}\sum_{k=1}^{\infty}\frac{(k+1)(k+2)(k+3)-6}{(k+1)(k+3)}\left(\frac{k+3}{k+1}c_k-\frac{k+1}{k+3}c_{k+2}\right)^2,
\end{equation*}
\begin{equation*}
\int_{-1}^{1}(1-x^2)^{\frac{1}{2}}G(\mathcal{Q}_{3}-1)Gdx=\frac{\pi}{2}\sum_{k=1}^{\infty}\frac{k}{(k+1)^2}c_k^2.
\end{equation*}
\begin{equation*}
\int_{-1}^{1}(1-x^2)^{\frac{1}{2}}F(\mathcal{Q}_{3}-1)Fdx=\frac{\pi}{2}\sum_{k=1}^{\infty}\frac{k}{(k+1)^2}f_k^2
\end{equation*}
and
\begin{equation*}
\int_{-1}^{1}(1-x^2)^{\frac{1}{2}}F^2dx=\frac{\pi}{2}\sum_{k=0}^{\infty}\frac{1}{(k+1)^2}f_k^2.
\end{equation*}

Then we have
\begin{equation*}
\begin{aligned}
&\frac{3}{5}\int_{-1}^{1}(1-x^2)^{\frac{3}{2}}G(\mathcal{Q}_{5}-6)Gdx- \frac{27}{16}\int_{-1}^{1}(1-x^2)^{\frac{1}{2}}G(\mathcal{Q}_{3}-1)Gdx \\
-&\frac{9}{16}\int_{-1}^{1}(1-x^2)^{\frac{1}{2}}F(\mathcal{Q}_{3}-1)Fdx-\frac{3}{4}\int_{-1}^{1}(1-x^2)^{\frac{1}{2}}F^2dx\\
=&\frac{3\pi}{8}\sum_{k=1}^{\infty}\frac{k(k^2+k+36)}{80}\left(\frac{c_k}{k+1}-\frac{c_{k+2}}{k+3}\right)^2+\frac{3\pi}{20}c_1^2+\pi\sum_{k=2}^{\infty}\frac{9k+114}{160(k+1)^2}c_k^2\geq 0.
\end{aligned}
\end{equation*}
\end{proof}

With this spectral inequality, we can prove Proposition \ref{seminormodd} when $N=5$.
\begin{proof}[Proof of Proposition \ref{seminormodd} when $N=5$.]
By definition of $P_{r,5}$, we have
\begin{equation*}
(1-x^2)P_{r,5}u=9P_{r,3}u-P_{r,3}((1-x^2)G').
\end{equation*}
Then by Lemma \ref{descent} we have
\begin{equation*}
\begin{aligned}
I_5
=&\int_{-1}^{1}(1-x^2)^{\frac{3}{2}}G^2 P_{r,5}udx\\
=&9\int_{-1}^{1}(1-x^2)^{\frac{1}{2}}G^2 P_{r,3}udx-\int_{-1}^{1}(1-x^2)^{\frac{1}{2}}G^2 P_{r,3}((1-x^2)G')dx\\
=&9\int_{-1}^{1}(1-x^2)^{\frac{1}{2}}G^2 P_{r,3}udx+2\int_{-1}^{1}(1-x^2)^{\frac{3}{2}}G G'P_{r,3}Gdx.
\end{aligned}
\end{equation*}

  Applying \eqref{DG1G2G3} to $(G,G,G)$, we have
\begin{equation*}
\int_{-1}^{1}(1-x^2)^{\frac{1}{2}}G^2 P_{r,3}udx=\frac{1}{3\pi}\mathcal{D}(G,G,G),
\end{equation*}
while applying  \eqref{DG1G2G3} to the triple $(u,G,G)$ gives
\begin{equation*}
\int_{-1}^{1}(1-x^2)^{\frac{3}{2}}G G'P_{r,3}Gdx=\frac{1}{2\pi} \mathcal{D}(G,F,F)-\frac{1}{2}\int_{-1}^{1}(1-x^2)^{\frac{1}{2}}F^2 P_{r,3}u dx.
\end{equation*}
Here $F:=(1-x^2)G'$.
It follows that
\begin{equation}\label{I5DGG}
I_5=\frac{3}{\pi}\mathcal{D}(G,G,G)+\frac{1}{\pi }\mathcal{D}(G,F,F)-\int_{-1}^{1}(1-x^2)^{\frac{1}{2}}F^2 P_{r,3}u dx.
\end{equation}

By Lemma \ref{GPbound}, we have
\begin{equation*}
\begin{aligned}
I_5
\leq &\frac{51}{20\pi\alpha}\int_{-1}^{1}\int_{-1}^{1}(1-x^2)^{\frac12}(1-y^2)^{\frac12}(\delta G)^2 dxdy+\frac{17}{20\pi \alpha}\int_{-1}^{1}\int_{-1}^{1}(1-x^2)^{\frac12}(1-y^2)^{\frac12}(\delta F)^2 dxdy\\
&+\frac{2}{\alpha}\int_{-1}^{1}(1-x^2)^{\frac{1}{2}}F^2 dx.    
\end{aligned}
\end{equation*}

Applying Lemma \ref{deltaH2} to $G$ and $F$ respectively, we get
\begin{equation*}
\begin{aligned}
I_5
\leq &\frac{51}{10\alpha}\int_{-1}^{1}(1-x^2)^{\frac12}G(\mathcal{Q}_{3}G-G) dx+\frac{17}{10 \alpha}\int_{-1}^{1}(1-x^2)^{\frac12}(\mathcal{Q}_{3}F-F)F dx\\
&+\frac{2}{\alpha}\int_{-1}^{1}(1-x^2)^{\frac{1}{2}}F^2 dx.    
\end{aligned}
\end{equation*}

Multiplying Lemma \ref{Q5Q3} by $\frac{136}{45\alpha}$ and plugging it into the estimate above, we have
\begin{equation*}
I_5\leq \frac{136}{75\alpha}\int_{-1}^{1}(1-x^2)^{\frac32}G(\mathcal{Q}_{5}G-6G) dx
\end{equation*}
as desired.
\end{proof}

Finally, we focus on $N=7$. Similarly, we need the following spectral inequality.
\begin{lemma}\label{Q7Q3}
For any smooth function $G$, set $F=(1-x^2)G'$ and $H=(1-x^2)F'$. Then we have
\begin{equation*}
\begin{aligned}
\frac{12}{5}\int_{-1}^{1}(1-x^2)^{\frac{5}{2}}G(\mathcal{Q}_{7}-120)Gdx 
\geq& \frac{675}{7}\int_{-1}^{1}(1-x^2)^{\frac{1}{2}}G(\mathcal{Q}_{3}-1)Gdx +\frac{306}{7}\int_{-1}^{1}(1-x^2)^{\frac{1}{2}}F(\mathcal{Q}_{3}-1)Fdx\\
&+\frac{9}{7}\int_{-1}^{1}(1-x^2)^{\frac{1}{2}}H(\mathcal{Q}_{3}-1)Hdx+32\int_{-1}^{1}(1-x^2)^{\frac{1}{2}}F^2dx\\
&+2\int_{-1}^{1}(1-x^2)^{\frac{1}{2}}H^2dx+48\int_{-1}^{1}(1-x^2)^{\frac{3}{2}}F^2dx.
\end{aligned}
\end{equation*}
\end{lemma}
\begin{proof}
To avoid ambiguity, we will keep the superscripts of the Gegenbauer polynomials $F_k^1$, $F_k^2$ and $F_k^3$ in this proof.

We first expand $G=\sum_{k=0}^{\infty}c_kF_k^1$. By the recurrence formulas \eqref{recurence1} and \eqref{recurence2}, we have
\begin{equation*}
F_k^1=\frac{(k+2)(k+3)}{6(k+1)}F_k^2-\frac{k(k-1)}{6(k+1)}F_{k-2}^2,
\end{equation*}
\begin{equation*}
F_k^1=\frac{(k+3)(k+4)(k+5)}{60(k+1)}F_k^3-\frac{(k-1)(k+3)}{30}F_{k-2}^3+\frac{(k-3)(k-2)(k-1)}{60(k+1)}F_{k-4}^3
\end{equation*}
and
\begin{equation*}
(1-x^2)(F_k^1)'=\frac{k(k+2)}{2(k+1)}(F_{k-1}^1-F_{k+1}^1).
\end{equation*}
It follows that
\begin{equation*}
\begin{aligned}
G=&\sum_{k=0}^{\infty}\frac{k+2}{6}\left(\frac{k+3}{k+1}c_k-\frac{k+1}{k+3}c_{k+2}\right)F_k^2\\
=&\sum_{k=0}^{\infty}\left(\frac{(k+3)(k+4)(k+5)}{60(k+1)}c_k-\frac{(k+1)(k+5)}{30}c_{k+2}+\frac{(k+1)(k+2)(k+3)}{60(k+5)}c_{k+4}\right)F_k^3\\
=:&\sum_{k=0}^{\infty}d_kF_k^3,
\end{aligned}
\end{equation*}
\begin{equation*}
\begin{aligned}
F:=&\sum_{k=0}^{\infty}f_kF_k^1=\sum_{k=0}^{\infty}\left(\frac{(k+1)(k+3)}{2(k+2)}c_{k+1}-\frac{(k-1)(k+1)}{2k}c_{k-1}\right)F_k^1\\
:=&\sum_{k=0}^{\infty}e_kF_k^2=\sum_{k=0}^{\infty}\left(\frac{(k+2)(k+3)}{6(k+1)}f_{k}-\frac{(k+1)(k+2)}{6(k+3)}f_{k+2}\right)F_k^2
\end{aligned}
\end{equation*}
where all coefficients with negative index are understood to be $0$, and
\begin{equation*}
H:=\sum_{k=0}^{\infty}h_kF_k^1=\sum_{k=0}^{\infty}\left(\frac{(k+1)(k+3)}{2(k+2)}f_{k+1}-\frac{(k-1)(k+1)}{2k}f_{k-1}\right)F_k^1
\end{equation*}
with the same convention.

Since
\begin{equation*}
\mathcal{Q}_{7} F_k^3=(k+1)(k+2)(k+3)(k+4)(k+5)F_k^3,
\end{equation*}
by \eqref{Fknorm}, we have
\begin{equation*}
\int_{-1}^{1}(1-x^2)^{\frac{5}{2}}G(\mathcal{Q}_{7}-120)Gdx=\frac{225\pi}{2}\sum_{k=1}^{\infty}\frac{(k+1)(k+2)(k+3)(k+4)(k+5)-120}{(k+1)(k+2)(k+3)^2(k+4)(k+5)}d_k^2,
\end{equation*}
\begin{equation*}
\int_{-1}^{1}(1-x^2)^{\frac{1}{2}}G(\mathcal{Q}_{3}-1)Gdx=\frac{\pi}{2}\sum_{k=1}^{\infty}\frac{k}{(k+1)^2}c_k^2,
\end{equation*}
\begin{equation*}
\int_{-1}^{1}(1-x^2)^{\frac{1}{2}}F(\mathcal{Q}_{3}-1)Fdx=\frac{\pi}{2}\sum_{k=1}^{\infty}\frac{k}{(k+1)^2}f_k^2,
\end{equation*}
\begin{equation*}
\int_{-1}^{1}(1-x^2)^{\frac{1}{2}}H(\mathcal{Q}_{3}-1)Hdx=\frac{\pi}{2}\sum_{k=1}^{\infty}\frac{k}{(k+1)^2}h_k^2,
\end{equation*}
\begin{equation*}
\int_{-1}^{1}(1-x^2)^{\frac{1}{2}}F^2dx=\frac{\pi}{2}\sum_{k=0}^{\infty}\frac{1}{(k+1)^2}f_k^2,
\end{equation*}
\begin{equation*}
\int_{-1}^{1}(1-x^2)^{\frac{1}{2}}H^2dx=\frac{\pi}{2}\sum_{k=0}^{\infty}\frac{1}{(k+1)^2}h_k^2,
\end{equation*}
and
\begin{equation*}
\int_{-1}^{1}(1-x^2)^{\frac{3}{2}}F^2dx=\frac{9\pi}{2}\sum_{k=0}^{\infty}\frac{1}{(k+1)(k+2)^2(k+3)}e_k^2.
\end{equation*}

Then we have
\begin{equation*}
\begin{aligned}
&\frac{12}{5}\int_{-1}^{1}(1-x^2)^{\frac{5}{2}}G(\mathcal{Q}_{7}-120)Gdx-\frac{675}{7}\int_{-1}^{1}(1-x^2)^{\frac{1}{2}}G(\mathcal{Q}_{3}-1)Gdx -\frac{306}{7}\int_{-1}^{1}(1-x^2)^{\frac{1}{2}}F(\mathcal{Q}_{3}-1)Fdx\\
-&\frac{9}{7}\int_{-1}^{1}(1-x^2)^{\frac{1}{2}}H(\mathcal{Q}_{3}-1)Hdx-32\int_{-1}^{1}(1-x^2)^{\frac{1}{2}}F^2dx-2\int_{-1}^{1}(1-x^2)^{\frac{1}{2}}H^2dx-48\int_{-1}^{1}(1-x^2)^{\frac{3}{2}}F^2dx\\
=&\frac{5\pi}{16}\sum_{k=1}^{\infty}\left(X_k(\Delta r_k-\frac{3}{k+3}r_k)^2+Y_kr_k^2+Z_k\frac{c_k^2}{(k+1)^2}\right)\geq 0,
\end{aligned}
\end{equation*}
where
\begin{equation*}
X_k=\frac{(k+3)(39k^4+560k^3+3375k^2+9310k+8136)}{350},
\end{equation*}
\begin{equation*}
Y_1=\frac{4701}{70},\qquad Y_2=\frac{37116}{175},\qquad Y_k=\frac{9(k^4+36k^3+629k^2+1442k-168)}{70(k+3)},\ k\geq 3,
\end{equation*}
\begin{equation*}
Z_1=\frac{5064}{35},\qquad Z_2=\frac{1632}{7},\qquad Z_3=\frac{9312}{35},\qquad Z_k=\frac{72(9k+100)}{35},\ k\geq 4,
\end{equation*}
and
\begin{equation}
r_k=\frac{c_k}{k+1}-\frac{c_{k+2}}{k+3},\qquad \Delta r_k=r_k-r_{k+2},
\end{equation}
which is the desired result.
\end{proof}

With this spectral inequality, we can prove Proposition \ref{seminormodd} when $N=7$.
\begin{proof}[Proof of Proposition \ref{seminormodd} when $N=7$.]
By definition of $P_{r,7}$, we have
\begin{equation*}
(1-x^2)P_{r,7}u=25P_{r,5}u-P_{r,5}((1-x^2)G').
\end{equation*}
Set $F=(1-x^2)G'$ and $H=(1-x^2)F'$. Then by Lemma \ref{descent}, we have
\begin{equation}\label{I7I5}
\begin{aligned}
I_7
=&\int_{-1}^{1}(1-x^2)^{\frac{5}{2}}G^2 P_{r,7}udx=25I_5-\int_{-1}^{1}(1-x^2)^{\frac{3}{2}}G^2 P_{r,5}Fdx\\
=&25I_5+2\int_{-1}^{1}(1-x^2)^{\frac{3}{2}}G FP_{r,5}Gdx.
\end{aligned}
\end{equation}

The factorization
\begin{equation*}
(1-x^2)P_{r,5}G=9P_{r,3}G-P_{r,3}((1-x^2)G')
\end{equation*}
and Lemma \ref{descent} to the second term with $(f_1,f_2,f_3)=(F,G,F)$ gives
\begin{equation}\label{GFP5G}
\int_{-1}^{1}(1-x^2)^{\frac{3}{2}}G FP_{r,5}Gdx=9 \int_{-1}^{1}(1-x^2)^{\frac{1}{2}}G FP_{r,3}Gdx+\int_{-1}^{1}(1-x^2)^{\frac{1}{2}}F^2P_{r,3}Fdx+\int_{-1}^{1}(1-x^2)^{\frac{1}{2}}G HP_{r,3}Fdx.   
\end{equation}

On the one hand, we apply Lemma \ref{descent} again to $(f_1,f_2,f_3)=(G,F,F)$ to get
\begin{equation}\label{F2P3F1}
\int_{-1}^{1}(1-x^2)^{\frac{1}{2}}F^2P_{r,3}Fdx=-2\int_{-1}^{1}(1-x^2)^{\frac{1}{2}}FHP_{r,3}Gdx.
\end{equation}

On the other hand, since
\begin{equation}\label{1-x2P3u}
(1-x^2)P_{r,5}u=9P_{r,3}u-P_{r,3}F,
\end{equation}
we have
\begin{equation}\label{F2P3F2}
\int_{-1}^{1}(1-x^2)^{\frac{1}{2}}F^2P_{r,3}Fdx=9\int_{-1}^{1}(1-x^2)^{\frac{1}{2}}F^2P_{r,3}udx-\int_{-1}^{1}(1-x^2)^{\frac{3}{2}}F^2P_{r,5}udx.
\end{equation}

Averaging \eqref{F2P3F1} and \eqref{F2P3F2}, we have 
\begin{equation}\label{F2P3F}
\int_{-1}^{1}(1-x^2)^{\frac{1}{2}}F^2P_{r,3}Fdx=-\int_{-1}^{1}(1-x^2)^{\frac{1}{2}}FHP_{r,3}Gdx+\frac{9}{2}\int_{-1}^{1}(1-x^2)^{\frac{1}{2}}F^2P_{r,3}udx-\frac{1}{2}\int_{-1}^{1}(1-x^2)^{\frac{3}{2}}F^2P_{r,5}udx.
\end{equation}

By \eqref{DG1G2G3}, we have
\begin{equation}\label{GFP3G}
\int_{-1}^{1}(1-x^2)^{\frac{1}{2}}G FP_{r,3}Gdx=\frac{1}{2\pi}\mathcal{D}(G,F,F)-\frac{1}{2}\int_{-1}^{1}(1-x^2)^{\frac{1}{2}} F^2P_{r,3}udx
\end{equation}
and
\begin{equation}\label{GHP3F}
\int_{-1}^{1}(1-x^2)^{\frac{1}{2}}G HP_{r,3}Fdx=\frac{1}{2\pi}\mathcal{D}(G,H,H)-\frac{1}{2}\int_{-1}^{1}(1-x^2)^{\frac{1}{2}} H^2P_{r,3}udx.
\end{equation}

Substituting \eqref{F2P3F}, \eqref{GFP3G} and \eqref{GHP3F} into \eqref{GFP5G}, by \eqref{I5DGG} and \eqref{I7I5}, we get
\begin{equation*}
\begin{aligned}
I_7
=&\frac{75}{\pi}\mathcal{D}(G,G,G)+\frac{34}{\pi}\mathcal{D}(G,F,F)+\frac{1}{\pi}\mathcal{D}(G,H,H)-25\int_{-1}^{1}(1-x^2)^{\frac{1}{2}}F^2P_{r,3}udx\\
-&\int_{-1}^{1}(1-x^2)^{\frac{1}{2}}H^2P_{r,3}udx-2\int_{-1}^{1}(1-x^2)^{\frac{1}{2}}FHP_{r,3}Gdx-\int_{-1}^{1}(1-x^2)^{\frac{3}{2}}F^2P_{r,5}udx.
\end{aligned}
\end{equation*}

Combining \eqref{F2P3F} and \eqref{1-x2P3u}, we get
\begin{equation*}
2\int_{-1}^{1}(1-x^2)^{\frac{1}{2}}FHP_{r,3}Gdx=-\int_{-1}^{1}(1-x^2)^{\frac{1}{2}}F^2(9P_{r,3}u-(1-x^2)P_{r,5}u)dx.
\end{equation*}

Then we have
\begin{equation*}
\begin{aligned}
I_7
=&\frac{75}{\pi}\mathcal{D}(G,G,G)+\frac{34}{\pi}\mathcal{D}(G,F,F)+\frac{1}{\pi}\mathcal{D}(G,H,H)-16\int_{-1}^{1}(1-x^2)^{\frac{1}{2}}F^2P_{r,3}udx\\
-&\int_{-1}^{1}(1-x^2)^{\frac{1}{2}}H^2P_{r,3}udx-2\int_{-1}^{1}(1-x^2)^{\frac{3}{2}}F^2P_{r,5}udx.
\end{aligned}
\end{equation*}

Applying Lemma \ref{GPbound} and Lemma \ref{deltaH2}, we obtain
\begin{equation*}
\begin{aligned}
I_7
\leq& \frac{675}{7\alpha}\int_{-1}^{1}(1-x^2)^{\frac{1}{2}}G(\mathcal{Q}_{3}G-G)dx+\frac{306}{7\alpha}\int_{-1}^{1}(1-x^2)^{\frac{1}{2}}F(\mathcal{Q}_{3}F-F)dx+\frac{9}{7\alpha}\int_{-1}^{1}(1-x^2)^{\frac{1}{2}}H(\mathcal{Q}_{3}H-H)dx\\
+&\frac{32}{\alpha}\int_{-1}^{1}(1-x^2)^{\frac{1}{2}}F^2 dx
+\frac{2}{\alpha}\int_{-1}^{1}(1-x^2)^{\frac{1}{2}}H^2 dx
+\frac{48}{\alpha}\int_{-1}^{1}(1-x^2)^{\frac{3}{2}}F^2 dx.
\end{aligned}
\end{equation*}

By Lemma \ref{Q7Q3}, we get
\begin{equation*}
I_7\leq \frac{12}{5\alpha}\int_{-1}^{1}(1-x^2)^{\frac{5}{2}}G(\mathcal{Q}_{7}G-120G)dx
\end{equation*}
as desired.
\end{proof}

\section{Proofs of the rigidity Theorems}\label{ProofMain}
In this section, we will prove Theorem \ref{main} and Theorem \ref{mainodd} respectively.

\subsection{Proof of Theorem \ref{main}}\label{subsec proof main}
We first prove Theorem \ref{main}. Note that if $\beta=0$, then $[(1-x^2)^{\frac{N}{2}-1}G]^{(\frac{N}{2}-1)}$ is identically zero by \eqref{D(N-2/2)G}. It follows that $(1-x^2)^{\frac{N}{2}-1}G$ is a polynomial of degree at most $\frac{N}{2}-2$. Since $G(-1)=G(1)=0$, $(1-x^2)^{\frac{N}{2}-1}G$ has a zero of multiplicity $\frac{N}{2}$ at $x=\pm 1$, which forces $G$ to be identically zero as desired. Suppose for contradiction that $\beta\neq 0$.  Then by \eqref{a} and \eqref{D(N-2/2)G}, we have $0<\beta<\frac{1}{\alpha}$.

To begin with, for an even number $N$, we introduce the quantity
\begin{equation*}%\label{def DN}
D_N=\sum\limits_{k=2}^{\infty}\left[\frac{\Gamma(k+N-1)}{\Gamma(k+1)}\lambda_k-(\lambda_2-\frac{c_N}{\alpha})\frac{\Gamma(k+N-1)}{\Gamma(k+1)}-\frac{R_{N,k}}{\alpha}\right]b_{k}^{2},
\end{equation*}
where 
\begin{equation*}
c_N=\frac{2(N^3+N^2+5N-4)}{9N}.
\end{equation*}

We also recall $M=\frac{N}{2}-1$ and
\begin{equation*}
R_{N,k}=N!+N!\sum_{r=1}^{M}\frac{M!}{3r!(r-1)!(M+r)!}\sum_{d=0}^{M-r}\frac{(r+d-1)!(M-r-d+1)[3(M+r)+d]}{d!(M+r+d)(M+r+d+1)}\prod_{i=0}^{r-1}(\lambda_k-\lambda_i ).
\end{equation*}
from \eqref{RNk}.

In fact, by Proposition \ref{RNkest} we have
\begin{equation*}
R_{N,k}\leq c_N\frac{\Gamma(k+N-1)}{\Gamma(k+1)},\qquad k\geq 2.
\end{equation*}

Then we have
\begin{equation*}
D_N= \sum\limits_{k=2}^{\infty}\left[(\lambda_k-\lambda_2)\frac{\Gamma(k+N-1)}{\Gamma(k+1)}+\frac{1}{\alpha}\left(c_N \frac{\Gamma(k+N-1)}{\Gamma(k+1)}-R_{N,k}\right)\right]b_{k}^{2}\geq 0.
\end{equation*}

On the other hand, since 
\begin{equation*}
R_{N,1}=\Gamma(N)\frac{N(N+1)}{3},
\end{equation*}
we have
\begin{equation*}
D_N=\sum\limits_{k=1}^{\infty}\left[\frac{\Gamma(k+N-1)}{\Gamma(k+1)}\lambda_k-(\lambda_2-\frac{c_N}{\alpha})\frac{\Gamma(k+N-1)}{\Gamma(k+1)}-\frac{R_{N,k}}{\alpha}\right]b_{k}^{2}+\Gamma(N)\left[N+2+\frac{1}{\alpha}\left(\frac{N(N+1)}{3}-c_N\right)\right]b_1^2.
\end{equation*}

By \eqref{Paneitzeigen} and the Gegenbauer expansion of $G$, we get the following spectral identities.

\begin{equation*}
\lfloor G\rfloor^2=\sum\limits_{k=1}^{\infty}\frac{\Gamma(k+N-1)}{\Gamma(k+1)}\lambda_k b_k^2,
\end{equation*}
\begin{equation*}
\int_{-1}^{1}|[(1-x^2)^{\frac{N}{2}-1} G]^{(\frac{N}{2}-1)}|^2 dx=\sum\limits_{k=1}^{\infty}\frac{\Gamma(k+N-1)}{\Gamma(k+1)}b_k^2,
\end{equation*}
\begin{equation*}
\int_{-1}^{1}(1-x^2)^{\frac{N}{2}-1+j}(G^{(j)})^2 dx=\sum\limits_{k=1}^{\infty}\left(\prod_{i=0}^{j-1}(\lambda_k-\lambda_i )\right)b_k^2, \qquad 1\leq j\leq \frac{N}{2}-1.
\end{equation*}

It then follows from Proposition \ref{Gfloor est} that
\begin{equation*}
\begin{aligned}
    D_N
    =&\lfloor G\rfloor^2-(\lambda_2-\frac{c_N}{\alpha})\int_{-1}^{1}|[(1-x^2)^\frac{N-2}{2} G]^{(\frac{N-2}{2})}|^2 dx-\frac{1}{\alpha}\sum_{k=1}^{\infty}R_{N,k}b_k^2+\Gamma(N)\left[N+2+\frac{1}{\alpha}\left(\frac{N(N+1)}{3}-c_N\right)\right]b_1^2\\
    \leq &-\left(\frac{N^2+3N+4}{2}-\frac{c_N}{\alpha}\right)\int_{-1}^{1}|[(1-x^2)^\frac{N-2}{2} G]^{(\frac{N-2}{2})}|^2+\Gamma(N)\left[N+2+\frac{1}{\alpha}\left(\frac{N(N+1)}{3}-c_N\right)\right]b_1^2.
\end{aligned}
\end{equation*}

By \eqref{D(N-2/2)G} and \eqref{Fknorm}, we have
\begin{equation*}
D_N\leq \frac{2^N\Gamma(\frac{N}{2})^2}{(N+1)(N-1)}\beta\left[\frac{N-1}{2}\left(N+2+\frac{1}{\alpha}\left(\frac{N(N+1)}{3}-c_N\right)\right)\beta-\left(\frac{N^2+3N+4}{2}-\frac{c_N}{\alpha}\right)\left(N+1-\frac{1}{\alpha}\right)\right].
\end{equation*}

Since $D_N\geq 0$, we get
\begin{equation*}
\beta \geq \frac{2}{N-1}\left(\frac{N^2+3N+4}{2}-\frac{c_N}{\alpha}\right)\left(N+1-\frac{1}{\alpha}\right)\left[N+2+\frac{1}{\alpha}\left(\frac{N(N+1)}{3}-c_N\right)\right]^{-1},
\end{equation*}
and hence
\begin{equation*}
\alpha\beta\geq \frac{2}{N-1}\left(\frac{N^2+3N+4}{2}\alpha-c_N \right)\left((N+1)\alpha-1\right)\left[(N+2)\alpha+\frac{N(N+1)}{3}-c_N\right]^{-1}=:Z_{N}(\alpha).
\end{equation*}

One can check that $Z_N(\alpha)$ is monotone increasing in $\alpha$. Hence,
\begin{equation*}
\alpha\beta\geq Z_{N}(\alpha)\geq Z_{N}(\frac{1}{2})=\left(\frac{N^2+3N+4}{2}-2c_N \right)\left[N+2+\frac{2N(N+1)}{3}-2c_N\right]^{-1},
\end{equation*}
and hence
\begin{equation*}
a=\frac{N}{N+1}(1-\alpha\beta)\leq \frac{N^2}{6(N+2+\frac{2N(N+1)}{3}-2c_N)}=\frac{3N^3}{2(2N^3+11N^2-2N+16)}.
\end{equation*}

Direct computation shows that
\begin{equation*}
a<\frac{N-2}{N-1}.
\end{equation*}
This contradicts Theorem \ref{bk} and completes the proof of Theorem \ref{main}.

\subsection{Proof of Theorem \ref{mainodd}}
In this subsection, we will prove Theorem \ref{mainodd}. Note that if $\beta=0$, by \eqref{D(N-2/2)Godd}, we have
\begin{equation*}
\sum_{k=1}^{\infty}\frac{\Gamma(k+N-1)}{\Gamma(k+1)}b_k^2=\int_{-1}^{1}(1-x^2)^{\frac{N-2}{2}}G\mathcal{Q}_{N}G dx=0.
\end{equation*}
It follows that $a_k=0$ for all $k\geq 1$. Since $a_0=0$, we deduce that $G$ is identically zero and $u$ is constant.

In the following, we will assume $\beta>0$. Let $t=\alpha\beta>0$. By Theorem \ref{bk}, we have
\begin{equation*}
0<t\leq \frac{2}{N(N-1)}.
\end{equation*}

Since
\begin{equation*}
\lfloor G\rfloor^2=\sum\limits_{k=1}^{\infty}\frac{\Gamma(k+N-1)}{\Gamma(k+1)}\lambda_k b_k^2,
\end{equation*}
and
\begin{equation*}
\int_{-1}^{1}(1-x^2)^{\frac{N}{2}-1}G^2=\sum\limits_{k=1}^{\infty}b_k^2,
\end{equation*}
by Proposition \ref{seminormodd}, we have
\begin{equation*}
\sum\limits_{k=1}^{\infty}\rho_{N,k}\frac{\Gamma(k+N-1)}{\Gamma(k+1)}b_k^2\leq 0,  
\end{equation*}
where
\begin{equation*}
\rho_{N,k}=\rho_{N,k}(\alpha):=\lambda_k+\frac{N(N-1)}{2}-\frac{N\theta_N}{\alpha}-\frac{N!-N\theta_N (N-2)!}{\alpha}\frac{\Gamma(k+1)}{\Gamma(k+N-1)}.
\end{equation*}

Since $\theta_3=\frac{2}{3}$, $\theta_5=\frac{136}{75}$, $\theta_7=\frac{12}{5}$, we have
\begin{equation*}
\rho_{N,k+1}-\rho_{N,k}=2k+N+\frac{(N!-N\theta_N (N-2)!)(N-2)}{\alpha (k+N-1)}\frac{\Gamma(k+1)}{\Gamma(k+N-1)}>0.
\end{equation*}
Hence $\rho_{N,k}$ is strictly increasing in $k$.

Define the auxiliary quantity
\begin{equation*}
D_N:=\sum_{k=2}^{\infty}(\rho_{N,k}-\rho_{N,2})\frac{\Gamma(k+N-1)}{\Gamma(k+1)}b_k^2.
\end{equation*}
Then we have $D_N\geq 0$.   

On the other hand, by \eqref{D(N-2/2)Godd} and \eqref{F1 beta}, we have
\begin{equation*}
\begin{aligned}
D_N
=&\sum_{k=1}^{\infty}(\rho_{N,k}-\rho_{N,2})\frac{\Gamma(k+N-1)}{\Gamma(k+1)}b_k^2+(\rho_{N,2}-\rho_{N,1})\Gamma(N)b_1^2\\
\leq &-\rho_{N,2}\sum_{k=1}^{\infty}\frac{\Gamma(k+N-1)}{\Gamma(k+1)}b_k^2+(\rho_{N,2}-\rho_{N,1})\Gamma(N)b_1^2\\
=&\frac{2^{N}\Gamma(\frac{N}{2})^2}{(N+1)(N-1)}\beta\left[-\rho_{N,2}\left(N+1-\frac{1}{\alpha}\right)+\frac{N-1}{2}(\rho_{N,2}-\rho_{N,1})\beta\right].
\end{aligned}
\end{equation*}

When $N=3$, we have $\rho_{3,1}=6-\frac{4}{\alpha}$ and $\rho_{3,2}=11-\frac{10}{3\alpha}$. Since $t\le \frac{1}{3}$, we have
\begin{equation*}
0\leq D_3\leq \frac{\pi t}{12\alpha^3}P_3(\alpha,t),
\end{equation*}
where
\begin{equation*}
    P_3(\alpha,t)=-132\alpha^2+15\alpha t+73\alpha+2t-10\le P_3(\alpha,\frac{1}{3})
 =-\frac23(6\alpha-1)(33\alpha-14)<0,
 \end{equation*}
 a contradiction.

When $N=5$, we have $\rho_{5,1}=15-\frac{59}{5\alpha}$ and $\rho_{5,2}=22-\frac{254}{25\alpha}$. Since $t\le \frac{1}{10}$, we have
\begin{equation*}
0\leq D_5\leq \frac{3\pi t}{50\alpha^3}P_5(\alpha,t),
\end{equation*}
where
\begin{equation*}
     P_5(\alpha,t)=-1650\alpha^2+175\alpha t+1037\alpha+41t-127
     \le P_5(\alpha,\frac{1}{10})
     =-\frac1{10}
 (16500\alpha^2-10545\alpha+1229)<0,
\end{equation*}
a contradiction.

When $N=7$, we have $\rho_{7,1}=28-\frac{21}{\alpha}$ and $\rho_{7,2}=37-\frac{18}{\alpha}$. Since $t\le \frac{1}{21}$, we have
\begin{equation*}
0\leq D_7\leq \frac{75\pi t}{8\alpha^3} P_7(\alpha,t),
\end{equation*}
where
\begin{equation*}
    P_7(\alpha,t)=-296\alpha^2+27\alpha t+181\alpha+9t-18\le
    P_7(\alpha,\frac{1}{21})=-\frac17(2072\alpha^2-1276\alpha+123)<0,
\end{equation*}
a contradiction.

This completes the proof of Theorem \ref{mainodd}.

\appendix

\section{Proof of Proposition \ref{newIBP}}\label{NewIBPproof}

\subsection{Notation}

Our strategy is to transform the problem from $(-1,1)$ to $\R$. To be more precise, we set
\[
x=\tanh t,
\qquad
w=1-x^2=\text{sech}^2t,
\qquad
D=\frac{d}{dx},
\qquad
T=\frac{d}{dt}=wD,
\qquad
R_a=T-2ax.
\]
By a slight abuse of notation, we use the same symbol for a function of \(x\) and its pullback under \(x=\tanh t\).  

In this section, we introduce the operators needed in the proof to deal with the complicated differentiation and integration by parts procedure.

For \(n\geq 0\), we define the differential operators
\[
V_{a,n}=R_aR_{a+1}\cdots R_{a+n-1},
\qquad
V_{a,0}=\Id,
\]
and put
\[
U_r=V_{1-r,r},
\qquad
W_s=V_{-s,s+1}V_{M-s+1,s}.
\]
The elementary conjugation rule
\begin{equation}\label{Ra conju}
R_a(w^jF)=w^jR_{a+j}F
\end{equation}
gives
\begin{equation}\label{eq:U}
U_r=w^rD^r.
\end{equation}

For any real number \(a\), we define
\[
L_a=-wD^2+2(a+1)xD=-w^{-a}D(w^{a+1}D).
\]
In particular, write
\[
L=L_{M+1},
\]
where we recall $M=\frac{N-2}{2}$, and define
\[
K_0=\Id,
\qquad
K_s=\prod_{\ell=1}^{s}
\bigl[L_M+\ell(2M+1-\ell)\bigr],
\qquad
A_s=DK_s.
\]
Then we have the following expressions of $K_s$ and $W_s$. In particular, we obtain the validity of \eqref{eq hat GM}.
\begin{lemma}\label{lem:darboux}
For all \(0\leq s\leq M\), we have
\begin{align}
    K_s&=(-1)^sD^s\bigl[w^{s-M}D^s(w^M\,\cdot\,)\bigr],
    \label{eq:Rodrigues}\\
    W_s&=(-1)^sw^{s+1}A_s.\label{eq:W-A}
\end{align}
Moreover,
\[
A_sG=\widehat G_s,
\qquad
A_M=(-1)^MD^{2M+1}(w^M\,\cdot\,).
\]
\end{lemma}

\begin{proof}
Introduce
\begin{equation}\label{def delta a}
    \delta_a=-wD-(a+1)w'=-w^{-a}D(w^{a+1}\,\cdot\,).
\end{equation}
Then by definition, \(L_a=\delta_aD\), and we have
\begin{equation}\label{eq:Ddelta}
    D^p\delta_a
    =\delta_{a+p}D^p+p(2a+p+1)D^{p-1}.
\end{equation}
In fact, the case \(p=1\) is
\(D\delta_a=\delta_{a+1}D+2(a+1)\), which can be checked directly, and the general case follows by induction.  Applying \eqref{eq:Ddelta} successively gives
\[
K_s=D^s\delta_{M-s}\delta_{M-s+1}\cdots\delta_{M-1}.
\]
Using the definition of $\delta_a$ in \eqref{def delta a}, we obtain \eqref{eq:Rodrigues}.

Since \(R_a=-\delta_{a-1}\), the same calculation gives
\[
V_{-s,s+1}=w^{s+1}D^{s+1},
\qquad
V_{M-s+1,s}=w^{s-M}D^s(w^M \cdot ).
\]
Combining these identities with \eqref{eq:Rodrigues} proves
\eqref{eq:W-A}.  Finally,
\[
DK_s
=\prod_{j=1}^{s}
\bigl[L+(j+1)(2M+2-j)\bigr]D,
\]
which is exactly the recursion defining \(\widehat G_s\).  Taking \(s=M\)
in \eqref{eq:Rodrigues} gives the last assertion.
\end{proof}

For \(r,d,s\geq 0\) with \(r+d+s=M\), we define
\[
A_{r,d}=\frac{M!}{r!d!s!}(r)_d(M+r)_d,
\qquad
B_{r,d}=\frac{M!}{r!d!s!}(r)_d(M+r+1)_d.
\]
Then a direct cancellation gives
\begin{equation}\label{eq:C-average}
C_{r,d}=\frac{2A_{r,d}+B_{r,d}}{3}.
\end{equation}

The idea of the proof is to establish a polarized version of \eqref{IN IBP}, so in the rest of the appendix, we study the properties of some trilinear forms defined below.

For \((r,d)=(0,0)\), and for
\(1\leq r\leq M\), \(0\leq d\leq M-r\), we define
\begin{equation}\label{def J rd}
J_{r,d}(f,g;h)
=(-1)^s\int_{\R}w^d(U_rf)(U_rg)(W_sh)\,dt,
\end{equation}
where $s=M-r-d$. We call $J_{0,0}(f,g;h)$ the top cell. The semicolon marks the distinguished third argument. Then we define the following trilinear forms:
\begin{align*}
\mathcal H_A(f,g;h)
&=J_{0,0}(f,g;h)
+\sum_{r=1}^{M}\sum_{d=0}^{M-r}A_{r,d}J_{r,d}(f,g;h), %\label{def HA}
\\ 
\mathcal H_B(f,g;h)
&=J_{0,0}(f,g;h)
+\sum_{r=1}^{M}\sum_{d=0}^{M-r}B_{r,d}J_{r,d}(f,g;h).%\label{def HB}
\end{align*}
Both forms are symmetric in their first two arguments. In the next section, we exploit additional symmetric structures of these trilinear forms, which play a central role in the proof.

\subsection{Symmetry results}
\subsubsection{The derivative-transfer identity}
We first give an identity relating $\mathcal H_B$ and $\mathcal H_A$.
\begin{proposition}\label{prop:transfer}
For all \(f,g,h\in C_c^\infty(\R)\),
\begin{equation}\label{eq:transfer}
    \mathcal H_B(f,g;Th)
    +\mathcal H_A(Tf,g;h)
    +\mathcal H_A(f,Tg;h)=0.
\end{equation}
\end{proposition}

\begin{proof}
For simplicity, we write
\[
s=M-r-d,
\qquad
F_r=U_rf,
\qquad
G_r=U_rg,
\qquad
H_s=W_sh,
\qquad
\varepsilon_s=(-1)^s,
\]
and put \(\Delta_{r,d}=B_{r,d}-A_{r,d}\).  Then a direct computation gives the following relations:
\begin{align}
    \Delta_{r,d+1}&=s(r+d)\,B_{r,d},\label{eq:coef1}\\
    (d+1)A_{r,d+1}&=(M+r)\Delta_{r,d+1},\label{eq:coef2}\\
    \Delta_{r,d}&=r(r+1)A_{r+1,d-1}\qquad(d\geq 1).
    \label{eq:coef3}
\end{align}
We also need the following commutator identities:
\begin{align}
    [T,U_r]&=r(r-1)wU_{r-1},\label{eq:comm-U}\\
    [T,W_s]&=-2s(M-s)\,wR_{(1-s)/2}W_{s-1}\text{ for }s\geq 1,
    \label{eq:comm-W}\\
    [T,W_0]&=0\notag,
\end{align}
where $[A,B]:=AB-BA$ denotes the commutator.
Indeed, using the facts that  \([T,R_a]=-2aw\) and \(R_aw=wR_{a+1}\), we obtain
\[
[T,V_{a,n}]=-n(2a+n-1)wV_{a+1,n-1}
\]
by induction. This gives \eqref{eq:comm-U}. The second identity \eqref{eq:comm-W} can be obtained similarly.

We now expand the left-hand side of \eqref{eq:transfer}. For a cell $r\ge 1$, using \eqref{eq:comm-U} and \eqref{eq:comm-W}, we find that after removing the integral of the total derivative
\(\varepsilon_sA_{r,d}T(w^dF_rG_rH_s)\), the remaining integrand has the following form
\begin{align}
    \varepsilon_sw^d\Bigl\{
    &-A_{r,d}\frac{T(w^d)}{w^d}F_rG_rH_s
    +\Delta_{r,d}F_rG_r(TH_s)-A_{r,d}r(r-1)w(F_{r-1}G_r+F_rG_{r-1})H_s\notag\\
    &+2s(r+d)\,B_{r,d}wF_rG_rR_{(1-s)/2}H_{s-1}
    \Bigr\}.\label{eq:residual}
\end{align}
Moreover, the top cell $(r,d)=(0,0)$ has no residual.

Fix a cell with \(s\geq 1\). We combine the last line of \eqref{eq:residual} with the first two terms belonging to the cell \((r,d+1,s-1)\).  By relations \eqref{eq:coef1}--\eqref{eq:coef2}, their sum is
\begin{equation}\label{eq:packet}
    \varepsilon_s\Delta_{r,d+1}w^{d+1}
    F_rG_rR_{2r+d+1}H_{s-1}.
\end{equation}
In \(L^2(dt)\), \(R_a^*=-R_{-a}\). Using the property \eqref{Ra conju} and the fact that
\[
R_{-2r}(F_rG_r)
=F_rR_{-r}G_r+G_rR_{-r}F_r,
\]
an integration-by-parts computation transforms
\eqref{eq:packet} into
\[
-\varepsilon_s\Delta_{r,d+1}w^{d+1}
(F_rG_{r+1}+F_{r+1}G_r)H_{s-1}.
\]
By \eqref{eq:coef3}, this exactly cancels the second line in \eqref{eq:residual} from the cell \((r+1,d,s-1)\). Therefore every residual appears in exactly one such packet. Hence the total residual is zero. In other words, the left-hand side of \eqref{eq:transfer} is an integral of a total derivative, so it is zero. This completes the proof.
\end{proof}

\subsubsection{The critical symmetric form}

In this subsection, we prove a symmetry statement needed for the defect between \(\mathcal H_B\) and \(\mathcal H_A\).  Throughout this subsection we assume \(M\geq 2\). For \(\ell\geq 0\), we define the polynomial
\[
P_\ell(z)=\prod_{\nu=2}^{\ell+1}
\bigl[z+\nu(2M+3-\nu)\bigr],
\qquad
P_0=1.
\]
For \(0\leq j\leq M-2\), define
\begin{equation}\label{eq:Sj}
S_j(z)=\sum_{e=0}^{M-2-j}\binom{M-2-j}{e}
(j+2)_e(M+j+2)_eP_{M-2-j-e}(z),
\end{equation}
and introduce the trilinear form
\begin{equation*}%\label{eq:Phi}
\Phi(F,G,H)
=\sum_{j=0}^{M-2}\binom{M-2}{j}
\int_{-1}^{1}w^{M+1+j}F^{(j)}G^{(j)}
S_j(L)H\,dx.
\end{equation*}

The following proposition is the central algebraic step of the proof. It shows that the form appearing naturally in the coefficient defect has more symmetry than is visible from its definition.

\begin{proposition}\label{prop:critical}
The trilinear form \(\Phi\) is invariant under every permutation of its
three arguments.
\end{proposition}

\begin{proof}
Define
\[
\mathcal R(F,G)
=w^{-M-1}\sum_{j=0}^{M-2}\binom{M-2}{j}(-D)^j
\bigl[w^{M+1+j}(S_j(L)F)G^{(j)}\bigr]
\]
and
\begin{align}
    \mathcal T(F,G,H)
    &=\int_{-1}^{1}w^{M+1}H\mathcal R(F,G)\,dx\notag\\
    &=\sum_{j=0}^{M-2}\binom{M-2}{j}
    \int_{-1}^{1}w^{M+1+j}(S_j(L)F)
    G^{(j)}H^{(j)}\,dx.\label{eq:Tcritical}
\end{align}
The second formula follows by integrating by parts, where we note that there are no boundary terms. It shows immediately that \(\mathcal T\) is symmetric in its last two arguments. Our strategy is to prove that \(\mathcal R(F,G)=\mathcal R(G,F)\); this will make \(\mathcal T\), and hence \(\Phi\), fully symmetric.

The first ingredient is a weighted Leibniz rule.  We state it for the
general operator \(L_a\).

\begin{lemma}\label{lem:weighted-Leibniz}
    Let \(a\in\R\), and suppose that \(L_aG=\eta G\) for some $\eta\in\R$.  Then, for every
    \(j\geq 0\),
    \begin{equation}\label{eq:weighted-Leibniz}
        w^{-a}(-D)^j\bigl[w^{a+j}FG^{(j)}\bigr]
        =\sum_{k=0}^{j}(-1)^k\binom jk
        \left(\prod_{h=k}^{j-1}
        \bigl[\eta-h(2a+h+1)\bigr]\right)
        w^kF^{(k)}G^{(k)}.
    \end{equation}
\end{lemma}

\begin{proof}
    Recall from \eqref{def delta a} that
    \[
    \delta_a=-wD+2(a+1)x=-w^{-a}D(w^{a+1}\,\cdot\,),
    \qquad
    L_a=\delta_aD.
    \]
    The commutation relation \eqref{eq:Ddelta} gives
    \begin{equation}\label{eq:DjLa}
        D^jL_a=\bigl[L_{a+j}+j(2a+j+1)\bigr]D^j,
    \end{equation}
    and direct telescoping gives
    \[
    w^{-a}(-D)^j(w^{a+j}\cdot)
    =\delta_a\delta_{a+1}\cdots\delta_{a+j-1}.
    \]
    We prove \eqref{eq:weighted-Leibniz} by induction on \(j\), simultaneously for all $a$.  The case \(j=0\) is immediate.  If the assertion holds at
    order \(j\), then \eqref{eq:DjLa} implies
    \[
    L_{a+j}G^{(j)}
    =\bigl[\eta-j(2a+j+1)\bigr]G^{(j)}.
    \]
    Using \(\delta_b(FQ)=F\delta_bQ-wF'Q\), we obtain
    \[
    \delta_{a+j}(FG^{(j+1)})
    =\bigl[\eta-j(2a+j+1)\bigr]FG^{(j)}
    -wF'G^{(j+1)}.
    \]
    Apply \(\delta_a\cdots\delta_{a+j-1}\) to this identity.  For the second
    term, we use
    \[
    \delta_a\cdots\delta_{a+j-1}(wQ)
    =w\,\delta_{a+1}\cdots\delta_{a+j}Q
    \]
    and apply the induction hypothesis with parameter \(a+1\) to the pair
    \(F',G'\).  Since
    \[
    L_{a+1}G'=(\eta-2a-2)G'
    \]
    and
    \[
    (\eta-2a-2)-h(2a+h+3)
    =\eta-(h+1)(2a+h+2),
    \]
    the two resulting sums have the same products after shifting the index in the second one.  Combining their binomial 	coefficients then gives \eqref{eq:weighted-Leibniz} at order \(j+1\).
\end{proof}

We apply the lemma with \(a=M+1\).  Suppose that
\[
LF=\xi F,
\qquad
LG=\eta G,
\]
where \(\xi\) and \(\eta\) denote eigenvalues.  Put
\[
\theta_h=h(2M+h+3).
\]
Then Lemma~\ref{lem:weighted-Leibniz} gives
\begin{equation}\label{eq:R-eigen}
    \mathcal R(F,G)
    =\sum_{k=0}^{M-2}w^kF^{(k)}G^{(k)}c_k(\xi,\eta),
\end{equation}
where
\begin{equation}\label{eq:Ck}
    c_k(\xi,\eta)=(-1)^k\sum_{j=k}^{M-2}
    \binom{M-2}{j}\binom jk
    S_j(\xi)\prod_{h=k}^{j-1}(\eta-\theta_h).
\end{equation}
It remains to prove that every coefficient \(c_k\) is symmetric in its two
variables.

Fix \(k\), set \(n=M-2-k\), and introduce the Newton polynomials on the grid \(\theta_h\):
\[
p_s(z)=\prod_{h=k}^{k+s-1}(z-\theta_h),
\qquad
p_0=1.
\]
Then \eqref{eq:Ck} becomes
\begin{equation}\label{eq:Ck-Newton}
    c_k(\xi,\eta)=(-1)^k\binom{M-2}{k}
    \sum_{r=0}^{n}\binom nr S_{k+r}(\xi)p_r(\eta).
\end{equation}
Write
\begin{equation}\label{eq:a-expansion}
    S_{k+r}(z)=\sum_{s=0}^{n-r}a_{r,s}p_s(z).
\end{equation}
To compute these coefficients, we use the connection formula
\begin{equation}\label{eq:connection}
    P_\ell(z)=\sum_{s=0}^{\ell}\binom\ell s
    (k+s+2)_{\ell-s}(2M+2+k+s-\ell)_{\ell-s}p_s(z).
\end{equation}
Indeed, for \(\varphi(q)=q(q+2M+3)\), we have the recursion formulas
\[
P_{\ell+1}(z)=[z-\varphi(-\ell-2)]P_\ell(z),
\qquad
zp_s(z)=p_{s+1}(z)+\varphi(k+s)p_s(z).
\]
If \(b_{\ell,s}\) denotes the coefficient of \(p_s\) on the right-hand
side of \eqref{eq:connection}, then these two relations give the recursion
\[
b_{\ell+1,s}=b_{\ell,s-1}
+\bigl[\varphi(k+s)-\varphi(-\ell-2)\bigr]b_{\ell,s}.
\]
It is straightforward to check that the displayed expression for \(b_{\ell,s}\) satisfies this recurrence, and the initial value is \(b_{0,0}=1\).
Thus \eqref{eq:connection} holds for every \(\ell\).

For fixed \(r,s\), put
\[
q=n-r-s,
\qquad
\alpha=k+r+2,
\qquad
\beta=k+s+2.
\]
Substituting \eqref{eq:connection} into \eqref{eq:Sj} and using
\[
\binom{n-r}{e}\binom{n-r-e}{s}
=\binom{n-r}{s}\binom qe,
\]
we obtain
\begin{equation}\label{eq:a-E}
    a_{r,s}=\binom{n-r}{s}E_q(\alpha,\beta),
\end{equation}
where
\begin{equation}\label{eq:E}
    E_q(\alpha,\beta)
    =\sum_{e=0}^{q}\binom qe
    (\alpha)_e(\alpha+M)_e
    (\beta)_{q-e}(\alpha+\beta+M+e)_{q-e}.
\end{equation}

It remains to prove the finite hypergeometric identity that
exchanges the two indices.

\begin{lemma}\label{lem:exchange}
    For every \(q\geq 0\) and all parameters for which \eqref{eq:E} is defined, we have
    \[
    E_q(\alpha,\beta)=E_q(\beta,\alpha).
    \]
\end{lemma}

\begin{proof}
    Recall the formal hypergeometric series
    \[
    {}_2F_1\!\left(\begin{matrix}a,b\\c\end{matrix};t\right)
    =\sum_{e=0}^{\infty}\frac{(a)_e(b)_e}{(c)_e e!}t^e.
    \]
    Then formula \eqref{eq:E} is equivalent to
    \begin{align}
        \frac{E_q(\alpha,\beta)}{q!}
        &=(\alpha+\beta+M)_q[t^q](1-t)^{-\beta}\cdot
        {}_2F_1\!\left(
        \begin{matrix}\alpha,\alpha+M\\\alpha+\beta+M\end{matrix};t
        \right).\label{eq:E-generating}
    \end{align}
    Here $[t^q]f(t)$ means the coefficient of $t^q$ in the Taylor expansion of $f(t)$.
    
    Euler's transformation
    \begin{equation}\label{eq:Euler}
        {}_2F_1\!\left(\begin{matrix}a,b\\c\end{matrix};t\right)
        =(1-t)^{c-a-b}
        {}_2F_1\!\left(\begin{matrix}c-a,c-b\\c\end{matrix};t\right)
    \end{equation}
    is valid as an identity of formal power series whenever the denominator parameters are defined.  Indeed, both sides have constant term $1$ and satisfy the same hypergeometric differential equation
    \[
    t(1-t)y''+[c-(a+b+1)t]y'-aby=0.
    \]
In our case, the open parameter region is $\alpha+\beta+M\notin\{0,-1,-2,\ldots\}$.
    
    Applying \eqref{eq:Euler} to \eqref{eq:E-generating} changes the series
    there into
    \[
    (1-t)^{-\alpha}
    {}_2F_1\!\left(
    \begin{matrix}\beta,\beta+M\\\alpha+\beta+M\end{matrix};t
    \right),
    \]
    which is the corresponding generating series for
    \(E_q(\beta,\alpha)\).

    It remains to remove the temporary restriction on parameters. In fact,  multiplying the coefficient identity by the common denominator, both sides become a polynomial in $\alpha$, $\beta$ and $M$. Thus the original equality must hold identically.  
\end{proof}

Since $\alpha$ and $\beta$ are positive integers in \eqref{eq:E}, by applying Lemma~\ref{lem:exchange} to \eqref{eq:a-E}, we obtain
\[
\binom nr a_{r,s}
=\frac{n!}{r!s!q!}E_q(\alpha,\beta)
=\binom ns a_{s,r}.
\]
Substituting this relation into \eqref{eq:Ck-Newton} and
\eqref{eq:a-expansion} proves $c_k(\xi,\eta)=c_k(\eta,\xi)$.
Thus \eqref{eq:R-eigen} is symmetric whenever \(F\) and \(G\) are
polynomial eigenfunctions of \(L\).

We now pass from eigenfunctions to arbitrary polynomials.  By \eqref{DE}, we know that
\begin{equation*}%\label{eq:Gegenbauer-eigen}
    LC_d^{(M+3/2)}
    =d(d+2M+3)C_d^{(M+3/2)}.
\end{equation*}
  Since \(C_d^{(M+3/2)}\) has degree \(d\), these
eigenpolynomials form a basis of the polynomial ring.  Bilinearity and
\eqref{eq:R-eigen} therefore imply
\[
\mathcal R(F,G)=\mathcal R(G,F)
\]
for all polynomial pairs. We now justify the passage to smooth functions. In fact, a repeated application of $L$ shows that the $k$-th Gegenbauer partial sum of a smooth function $F$ converges to $F$ in $C^\ell[-1,1]$ for any $\ell$. On the other hand, the formula defining $\mathcal R$ is a finite bilinear differential expression of order at most $2M$ with smooth polynomial coefficients.  It is therefore continuous from $C^{2M}\times C^{2M}$ to $C^0$.  Taking polynomial partial sums for both $F$ and $G$ proves $\mathcal{R}(F,G)=\mathcal{R}(G,F)$ for arbitrary smooth functions.  

It follows that \(\mathcal T\) is symmetric in its first two arguments,
and \eqref{eq:Tcritical} already shows symmetry in its last two arguments.
Therefore \(\mathcal T\) is fully symmetric.  Since $\Phi(F,G,H)=\mathcal T(H,F,G)$, \(\Phi\) is also fully symmetric as claimed.
\end{proof}

\subsection{The symmetric defect}

We next identify the difference between the two trilinear forms.  The
coefficient arrays agree on the bottom edge \(d=0\), and we will see that all remaining terms combine exactly into the form constructed in the preceding subsection.

We define the defect
\[
\Delta(f,g;h)=\mathcal H_B(f,g;h)-\mathcal H_A(f,g;h).
\]
The following lemma identifies the defect with the symmetric form.

\begin{lemma}\label{lem:defect}
If \(M\geq 2\), then
\begin{equation}\label{eq:defect-bridge}
    \Delta(f,g;h)=M(M-1)\Phi(f',g',h').
\end{equation}
For \(M=1\), \(\Delta=0\).
\end{lemma}

\begin{proof}
For \(d=0\), the definitions give \(B_{r,0}=A_{r,0}\).  For
\(d\geq 1\), direct cancellation gives
\begin{equation}\label{eq:defect-coef}
    \begin{aligned}
        B_{r,d}-A_{r,d}=&r(r+1)A_{r+1,d-1}\\
        =&M(M-1)\binom{M-2}{r-1}\binom{M-1-r}{d-1}
        (r+1)_{d-1}(M+r+1)_{d-1}.
    \end{aligned}
\end{equation}
By Lemma~\ref{lem:darboux}, every non-top cell has the following expression
in the \(x\)-variable:
\begin{equation*}%\label{eq:J-x}
    J_{r,d}(f,g;h)
    =\int_{-1}^{1}w^{M+r}f^{(r)}g^{(r)}DK_{M-r-d}h\,dx.
\end{equation*}
Consequently, only cells with \(d\geq1\) contribute to \(\Delta\).

The operator $DK_{M-r-d}$ has exactly the product occurring in \eqref{eq:Sj}.  Indeed, we note that \(D(p(L_M))=p(L+2M+2)D\) for every polynomial \(p\). Then for $s=M-r-d$, we have
\begin{align*}
    DK_s
    &=\prod_{\ell=1}^{s}
    \bigl[L+(\ell+1)(2M+2-\ell)\bigr]D\\
    &=\prod_{p=2}^{s+1}
    \bigl[L+p(2M+3-p)\bigr]D
    =P_s(L)D.
\end{align*}

Substituting these identities and \eqref{eq:defect-coef} into the defect
sum, we obtain
\begin{align*}
 \Delta(f,g;h)
 &=M(M-1)\sum_{j=0}^{M-2}\binom{M-2}{j}
 \int_{-1}^1w^{M+1+j}(f')^{(j)}(g')^{(j)}\\
 &\quad\times\left[\sum_{e=0}^{M-2-j}\binom{M-2-j}{e}
 (j+2)_e(M+j+2)_eP_{M-2-j-e}(L)\right]h'\,dx.
\end{align*}
The operator in square brackets is exactly $S_j(L)$ by definition. By letting $j=r-1$ and $e=d-1$, we see that this is exactly \(M(M-1)\Phi(f',g',h')\) by definition. This proves \eqref{eq:defect-bridge}. If $M=1$, there is no index pair with $d\ge1$, so the defect is zero.
\end{proof}

In summary, Proposition~\ref{prop:critical} and Lemma~\ref{lem:defect} show that \(\Delta\) is fully symmetric.

We close this section by proving a symmetry property of the following trilinear form:
\begin{equation}\label{eq:cyclic-form}
\mathcal C(f,g,h)
=\mathcal H_B(f,g;h)
+\mathcal H_A(g,h;f)
+\mathcal H_A(h,f;g).
\end{equation}

\begin{proposition}\label{prop:translation}
For all compactly supported smooth \(f,g,h\), we have
\begin{equation}\label{eq:translation}
    \mathcal C(Tf,g,h)+\mathcal C(f,Tg,h)+\mathcal C(f,g,Th)=0.
\end{equation}
\end{proposition}

\begin{proof}
By inserting \(\mathcal H_B=\mathcal H_A+\Delta\) into
Proposition~\ref{prop:transfer}, we obtain
\begin{equation*}%\label{eq:HA-translation}
    \mathcal H_A(Tf,g;h)+\mathcal H_A(f,Tg;h)+\mathcal H_A(f,g;Th)
    =-\Delta(f,g;Th).
\end{equation*}
We apply this identity to the three cyclic orderings of \((f,g,h)\). Then \eqref{eq:translation} follows from the full symmetry of \(\Delta\).
\end{proof}

Using Proposition \ref{prop:translation}, we can obtain a canonical form of $\mathcal{C}$.
\begin{corollary}\label{lem:canonical}
After integrations by parts, the trilinear form \(\mathcal C\) has a unique representation
\begin{equation}\label{eq:canonical}
	\mathcal C(f,g,h)
	=\int_{\R}h(t)\sum_{p,q}c_{p,q}
	T^pf(t)\,T^qg(t)\,dt
\end{equation}
	with constants \(c_{p,q}\) for compactly supported smooth \(f,g,h\).
\end{corollary}

\begin{proof}
Using integration by parts, we move every derivative from the third argument onto the coefficient and	the first two arguments.  This gives \eqref{eq:canonical} initially with
	smooth functions \(c_{p,q}(t)\).  This representation is unique. In fact, \(h\) can be supported in an arbitrarily small neighborhood of a fixed point $t_0$, and \(f\) and \(g\) can also be prescribed arbitrarily. By taking $f$ and $g$ in the form $(t-t_0)^r\chi(t-t_0)$ for some power $r$ and a standard cutoff function $\chi$, we see that the coefficients $c_{p,q}(t)$ are uniquely determined.
	
	Now we substitute the canonical expression \eqref{eq:canonical} into the  identity \eqref{eq:translation}.	After one integration by parts, we obtain that
	\begin{align*}
	0=&\int_{\R}T\left(h(t)\sum_{p,q}c_{p,q}
	T^pf(t)\,T^qg(t)\right)\,dt-	\int_{\R}h(t)\sum_{p,q}c_{p,q}'(t)
		T^pf(t)\,T^qg(t)\,dt\\
		=&-	\int_{\R}h(t)\sum_{p,q}c_{p,q}'(t)
		T^pf(t)\,T^qg(t)\,dt.
	\end{align*}
	Thus the same argument of uniqueness of the representation gives that \(c'_{p,q}=0\) for all	\(p,q\).
\end{proof}

We then prove that all $c_{p,q}$ vanish.

\subsection{Conclusion of the proof}
In this subsection, we complete the proof of Proposition \ref{newIBP} by establishing the cyclic cancellation property of the trilinear form $\mathcal{C}$.
\begin{proposition}\label{prop:cyclic}
For every \(f,g,h\in C_c^\infty(\R)\),
\[
\mathcal C(f,g,h)=0.
\]
\end{proposition}
Our strategy is to take some truncated exponential functions as test functions $f, g$, and $h$, and then compute $\mathcal C(f,g,h)$ using both \eqref{eq:canonical} and its definition \eqref{eq:cyclic-form}. We first need the following finite cyclic identity.
\begin{lemma}\label{lem:finite}
If $a+b+c=0$, then
\begin{equation}\label{eq:finite}
    \sum_{\mathrm{cyc}(a,b,c)}\sum_{r=0}^{M}
    (-1)^r\binom{M}{r}
    (a)_r(b)_r(c)_{M-r+1}(c-M)_{M-r}=0.
\end{equation}
Here the cyclic triples are $(a,b,c)$, $(b,c,a)$, and $(c,a,b)$.
\end{lemma}

\begin{proof}
Define
\begin{equation}\label{eq:Ffirst}
    F(a,b,c)=\sum_{r=0}^{M}\binom{M}{r}
    (a)_r(b)_r(1+c)_{M-r}(1-c+r)_{M-r}.
\end{equation}
This expression is symmetric in $a,b$.  We show that it is also symmetric
in $b,c$.

Write
\[
{}_2F_1\!\left(\begin{matrix}A,B\\C\end{matrix};t\right)
=\sum_{n=0}^{\infty}\frac{(A)_n(B)_n}{(C)_n n!}t^n.
\]
Assume temporarily that the denominator parameters do not vanish.
Coefficient extraction gives
\begin{align*}
    F(a,b,c)
    &=M!(1-c)_M[t^M](1-t)^{-1-c}
    {}_2F_1\!\left(\begin{matrix}a,b\\1-c\end{matrix};t\right)
    \notag\\
    &=M!(1-c)_M[t^M](1-t)^{-c}
    {}_2F_1\!\left(\begin{matrix}1+a,1+b\\1-c\end{matrix};t\right).
    %\label{eq:Eulerstep}
\end{align*}
The second equality follows from Euler's transformation
\[
{}_2F_1\!\left(\begin{matrix}A,B\\C\end{matrix};t\right)
=(1-t)^{C-A-B}
{}_2F_1\!\left(\begin{matrix}C-A,C-B\\C\end{matrix};t\right),
\]
using $a+b+c=0$.  Pfaff's transformation gives
\begin{equation}\label{eq:Pfaffstep}
    F(a,b,c)
    =M!(1-c)_M[t^M](1-t)^{b-1}
    {}_2F_1\!\left(
    \begin{matrix}1+a,a\\1-c\end{matrix};-\frac{t}{1-t}\right).
\end{equation}
For completeness, Pfaff's identity
\[
{}_2F_1\!\left(\begin{matrix}A,B\\C\end{matrix};t\right)
=(1-t)^{-A}
{}_2F_1\!\left(
\begin{matrix}A,C-B\\C\end{matrix};\frac{t}{t-1}\right)
\]
follows by comparing coefficients.  Indeed, the coefficient of $t^n$ on
its right-hand side is
\[
\frac{(A)_n}{n!}
\sum_{r=0}^{n}(-1)^r\binom nr\frac{(C-B)_r}{(C)_r}
=\frac{(A)_n(B)_n}{(C)_n n!},
\]
where the last equality is the finite Chu--Vandermonde identity.  One
proof of that identity is to insert, first in a region where the Beta
integral converges,
\[
\frac{(D)_r}{(C)_r}
=\frac{\Gamma(C)}{\Gamma(D)\Gamma(C-D)}
\int_0^1u^{D+r-1}(1-u)^{C-D-1}\,du
\]
and then use
$\sum_{r=0}^{n}(-1)^r\binom nr u^r=(1-u)^n$; after
clearing denominators, polynomial continuation
removes the temporary restrictions on the parameters.

Expanding \eqref{eq:Pfaffstep} now yields
\begin{equation}\label{eq:Fsecond}
    F(a,b,c)=\sum_{r=0}^{M}(-1)^r\binom{M}{r}
    (a)_r(a+1)_r
    (1-b+r)_{M-r}(1-c+r)_{M-r}.
\end{equation}
Thus $F$ is symmetric in $b,c$.  Together with the symmetry in $a,b$
visible in \eqref{eq:Ffirst}, this proves that $F$ is invariant under all
permutations of $a,b,c$.  Although the coefficient calculation was made
for generic parameters, both sides of \eqref{eq:Fsecond} are polynomials,
so the identity holds for all $a,b,c$ with $a+b+c=0$.

Finally, the left-hand side of \eqref{eq:finite} is
\[
(-1)^M\bigl[cF(a,b,c)+aF(b,c,a)+bF(c,a,b)\bigr]
=(-1)^M(a+b+c)F(a,b,c)=0.
\]
\end{proof}

The next lemma controls the trilinear form on truncated exponentials.

\begin{lemma}\label{lem:word}
Let $n\ge 1$. Let $z,\nu_1,\ldots,\nu_n\in\R$.  There are polynomials
$p_0,\ldots,p_n$ in $x$ and $w$ such that, for every smooth function $\varphi$,
\begin{equation}\label{eq:wordexpansion}
    e^{-2zt}R_{\nu_1}\cdots R_{\nu_n}
    \bigl(e^{2zt}\varphi(t)\bigr)
    =\sum_{k=0}^{n}p_k(x(t),w(t))T^k\varphi(t),
\end{equation}
and
\begin{equation}\label{eq:p0boundary}
    p_0(1,0)=2^n\prod_{j=1}^{n}(z-\nu_j).
\end{equation}
Consequently, as $t\to+\infty$,
\begin{equation}\label{eq:pureexp}
    e^{-2zt}R_{\nu_1}\cdots R_{\nu_n}(e^{2zt})
    =2^n\prod_{j=1}^{n}(z-\nu_j)+O(w(t)).
\end{equation}
\end{lemma}

\begin{proof}
For $n=1$, the result follows from the fact that
\begin{equation}\label{R n=1}
    e^{-2zt}R_{\nu_1}(e^{2zt}F)
    =TF+2(z-\nu_1x)F.
\end{equation}
Suppose \eqref{eq:wordexpansion} holds for $n$. Then, since
$Tx=w$ and $Tw=-2xw$,	applying $R_{\nu_1}$ preserves the form	\eqref{eq:wordexpansion} with polynomial coefficients. Moreover, at $(x,w)=(1,0)$, every $T$-derivative of a polynomial in $x,w$ vanishes.  Hence by \eqref{R n=1}, the value of the new coefficient of $\varphi$ is multiplied by $2(z-\nu_1)$.
Thus \eqref{eq:p0boundary} is proved by induction.

In particular, for $\varphi=1$, the right-hand side of \eqref{eq:wordexpansion} is
$p_0(x(t),w(t))$.  Since $p_0$ is a polynomial and
\[
x-1=-\frac{w}{1+x},
\]
we have	$p_0(x(t),w(t))-p_0(1,0)=O(w(t))$ for $t$ large.  This proves	\eqref{eq:pureexp}.
\end{proof}

Applying Lemma~\ref{lem:word} to the operators $U_r$ and $W_s$, we obtain
\begin{align}
U_r(e^{2at})
&=e^{2at}\bigl(2^r(a)_r+O(w(t))\bigr),
\label{eq:Uasymp}\\
W_s(e^{2ct})
&=e^{2ct}\bigl(2^{2s+1}(c)_{s+1}(c-M)_s+O(w(t))\bigr),
\label{eq:Wasymp}
\end{align}
where the error terms are uniform for $t\geq1$.

\begin{proof}[Proof of Proposition \ref{prop:cyclic}]
We define
\[
P(X,Y)=\sum_{p,q}c_{p,q}X^pY^q,
\]
which, in view of \eqref{eq:canonical}, is the symbol of $\mathcal C$ in its first two arguments.

For $R\geq2$, choose a cutoff function $\chi_R$ such that
\begin{equation*}
    \chi_R=1\ \text{on }[1,R+1],\qquad
    \operatorname{supp}\chi_R\subset[0,R+2],
\end{equation*}
and, for every $k\geq0$, the derivatives $T^k\chi_R$ are bounded uniformly in $R$.  Moreover, every positive-order derivative of $\chi_R$ vanishes outside
\[
[0,1]\cup[R+1,R+2].
\]
For example, fix a smooth function $\rho$ such that
$\rho=0$ on $(-\infty,0]$ and $\rho=1$ on $[1,\infty)$, and set
\[
\chi_R(t)=\rho(t)\rho(R+2-t).
\]

Fix arbitrary $a,b\in\R$ and set $c=-a-b$.  Define
\begin{equation*}
    f_R(t)=e^{2at}\chi_R(t),\qquad
    g_R(t)=e^{2bt}\chi_R(t),\qquad
    h_R(t)=e^{2ct}\chi_R(t).
\end{equation*}
These three functions are compactly supported.

We first evaluate \eqref{eq:canonical}.  On $[1,R+1]$,
$\chi_R(t)=1$, so
\[
T^pf_R=(2a)^pe^{2at},\qquad
T^qg_R=(2b)^qe^{2bt},\qquad
h_R=e^{2ct}.
\]
Since $a+b+c=0$, the integral over this interval is exactly
$RP(2a,2b)$.  On the two remaining intervals
\[
[0,1]\cup[R+1,R+2],
\]
we have
\[
T^pf_R
=e^{2at}(T+2a)^p\chi_R(t),
\]
and analogous formulas for $g_R$ and $h_R$.  The derivatives
of $\chi_R$ are uniformly bounded, the total length of these intervals is
$2$, and the exponential factors multiply to
$e^{2(a+b+c)t}=1$.  Therefore
\begin{equation}\label{eq:canonicalR}
    \mathcal C(f_R,g_R,h_R)
    =RP(2a,2b)+O(1),
\end{equation}
where the $O(1)$ term is uniform in $R\geq2$.

We now evaluate the same quantity from its definition \eqref{eq:cyclic-form}.  Lemma~\ref{lem:word} and the uniform bounds for the derivatives of $\chi_R$ show that, on the two intervals where the cutoff varies, every normalized expression
\[
e^{-2zt}R_{\nu_1}\cdots R_{\nu_n}
\bigl(e^{2zt}\chi_R(t)\bigr)
\]
is bounded uniformly in $R$. As before, in every term of \eqref{eq:cyclic-form}, the three exponential factors again multiply to $1$. Since $0<w\leq1$ and these two intervals have total length $2$, their total contribution to \eqref{eq:cyclic-form} is $O(1)$, uniformly in $R$.

It remains to calculate the integrals over $[1,R+1]$.  On this
interval the cutoffs are equal to $1$, and
\eqref{eq:Uasymp}--\eqref{eq:Wasymp} apply. We first consider a term with
$d=0$ and $s=M-r$.  Its integrand has the expansion
\begin{align}
    &(-1)^s(U_rf_R)(U_rg_R)(W_sh_R)=(-1)^M2^{2M+1}(-1)^r
    (a)_r(b)_r(c)_{M-r+1}(c-M)_{M-r}+O(w(t)).
    \label{eq:d0leading}
\end{align}
Thus as $t\to\infty$, the first term on the right-hand side of \eqref{eq:d0leading} is the leading term.  Every other term in the product expansion is $O(w(t))$.

For a term with $d\geq1$, the explicit
factor $w^d$ then gives
\begin{equation*}%\label{eq:dpositive}
    (-1)^sw^d(U_rf_R)(U_rg_R)(W_sh_R)
    =O(w(t)).
\end{equation*}
Hence every $d>0$ cell contributes only an $O(w(t))$ term on
$[1,R+1]$.

The top cell is the case $r=0$ of \eqref{eq:d0leading}.  Since  $A_{r,0}=B_{r,0}=\binom{M}{r}$, the contributions from $[1,R+1]$ to
$\mathcal H_A(f_R,g_R;h_R)$ and
$\mathcal H_B(f_R,g_R;h_R)$ are both
\begin{align}
    &R(-1)^M2^{2M+1}
    \sum_{r=0}^{M}(-1)^r\binom{M}{r}
    (a)_r(b)_r(c)_{M-r+1}(c-M)_{M-r}+O\!\left(\int_1^{R+1}w(t)\,dt\right).
    \label{eq:rooted}
\end{align}
Furthermore,
\begin{equation*}%\label{eq:wint}
    \int_1^{R+1}w(t)\,dt
    \leq\int_1^\infty w(t)\,dt
    =1-\tanh 1.
\end{equation*}
Thus the remainder term in \eqref{eq:rooted} is $O(1)$ uniformly in $R$.

Adding the three cyclic terms in \eqref{eq:cyclic-form} gives
\begin{align*}
    \mathcal C(f_R,g_R,h_R)
    =R(-1)^M2^{2M+1}
    \sum_{\mathrm{cyc}(a,b,c)}\sum_{r=0}^{M}
    (-1)^r\binom{M}{r}
    (a)_r(b)_r(c)_{M-r+1}(c-M)_{M-r}
    +O(1).
    %\label{eq:directR}
\end{align*}
By Lemma~\ref{lem:finite}, the coefficient of $R$ is zero.  Hence
\begin{equation}\label{eq:directbounded}
    \mathcal C(f_R,g_R,h_R)=O(1)
\end{equation}
with a bound independent of $R$.

Comparing \eqref{eq:canonicalR} and \eqref{eq:directbounded} and letting $R\to\infty$, we obtain
\[
P(2a,2b)=0.
\]
Since $a,b\in\R$ are arbitrary, $P$ is the zero polynomial.  Thus every
$c_{p,q}$ in \eqref{eq:canonical} is zero, and
\(\mathcal C(f,g,h)=0\) for all compactly supported smooth
$f,g,h$.
\end{proof}

In the last step, we show that the previous results for compactly supported functions can also be applied to $G$.

In fact, smoothness of $u(x)$ implies that $u'(x)$ and all of its $x$-derivatives are bounded on $[-1,1]$.  Since $T=wD$ and $Tw=-2xw$, an induction argument gives that for every $m\ge0$,
\begin{equation}
	T^mG=O(w)=O(e^{-2|t|})\qquad(|t|\to\infty).\label{eq:G-pole-decay}
\end{equation}
We recall that each operator $R_a=T-2ax$, $U_r$, and $W_s$ is a finite differential operator in $T$ with bounded coefficients that are polynomials in $x,w$.  Hence the same estimate holds after applying any of these operators.

We choose a cutoff function $\chi_R\in C_c^\infty(\R)$ with $\chi_R=1$ on $[-R,R]$, $\chi_R=0$ outside $[-R-1,R+1]$, and with all $T$-derivatives bounded independently of $R$.  Then $G_R:=\chi_RG\in C_c^\infty(\R)$. We show that for every cell $J_{r,d}$, we have
\begin{equation*}
		J_{r,d}(G_R,G_R,G_R)\to J_{r,d}(G,G,G), \quad R\to \infty.
\end{equation*}

  On the two intervals $(-R-1,-R)$ and $(R,R+1)$, Leibniz' rule and the uniform derivative bounds for $\chi_R$, together with \eqref{eq:G-pole-decay}, show that every term containing any derivative of $G_R-G$ is $O(e^{-2R})$. 

By definition \eqref{def J rd}, a cell $J_{r,d}$ in the summation of $\mathcal{H}_A$ or $\mathcal{H}_B$ is a product of some differentiated factors, multiplied by a bounded coefficient and at most a fixed power of $w$. Therefore its contribution on $(-R-1,-R)$ and $(R,R+1)$ is  $O(e^{-2R})$ because the intervals have total length two. By \eqref{eq:G-pole-decay}, the remaining tails are also dominated by $Ce^{-2R}$ and tend to zero.  Thus every cell converges as $R\to\infty$.

Now we take $f=g=h=\chi_RG$ in Proposition~\ref{prop:cyclic} and let $R\to \infty$.  By
\eqref{eq:C-average} and the previous discussion, we obtain
\[
 0=\mathcal C(G,G,G)
 =3\left(J_{0,0}(G,G;G)
 +\sum_{r=1}^{M}\sum_{d=0}^{M-r}C_{r,d}J_{r,d}(G,G;G)\right).
\]
For $s=M-r-d$, equations \eqref{eq:U} and \eqref{eq:W-A}, together with the facts that $dt=dx/w$ and $A_sG=\widehat G_s$, give
\[
 J_{r,d}(G,G;G)
 =\int_{-1}^{1}w^{M+r}\bigl(G^{(r)}\bigr)^2
 \widehat G_{M-r-d}\,dx.
\]
For the top cell, we also have
\[
 J_{0,0}(G,G;G)
 =(-1)^M\int_{-1}^{1}w^MG^2D^{2M+1}(w^MG)\,dx=-I_N.
\]
Recall that $M=\frac{N-2}{2}$. Then substitution of these two formulas completes the proof of Proposition~\ref{newIBP}.

\section{The coefficient bound for $R_{N,k}$}\label{RNkbound}

In this appendix, we prove the bounds for $R_{N,k}$ needed in Subsection \ref{subsec proof main}, which will be used in the proof of Theorem \ref{main}. Recall that 
\begin{equation*}
\begin{aligned}
    R_{N,k}=&N!+N!\sum_{r=1}^{M}\sum_{d=0}^{M-r}\frac{(M-r-d+1)!}{(M+r+d+1)!}C_{r,d}\prod_{i=0}^{r-1}(\lambda_k-\lambda_i ).\\
    =&N!+N!\sum_{r=1}^{M}\frac{M!}{3r!(r-1)!(M+r)!}\sum_{d=0}^{M-r}\frac{(r+d-1)!(M-r-d+1)[3(M+r)+d]}{d!(M+r+d)(M+r+d+1)}\prod_{i=0}^{r-1}(\lambda_k-\lambda_i ).    
\end{aligned}
\end{equation*}
Then we have the following estimate on $R_{N,k}$.
\begin{proposition}\label{RNkest}
For every $k\geq 2$, we have
\begin{equation*}
    \frac{\Gamma(k+1)}{\Gamma(k+N-1)}R_{N,k}\leq c_N:=\frac{2(N^3+N^2+5N-4)}{9N}.  
\end{equation*}
\end{proposition}

To prove this estimate, we need several technical lemmas.

\begin{lemma}\label{pikrest}
For any even $N\geq 6$, $k\geq 2$ and $0\leq r\leq \min\{k,M\}$, let
\begin{equation*}
    \pi_{k,0}:=  \frac{(N-2)!\Gamma(k+1)}{\Gamma(k+N-1)}  
\end{equation*}
and
\begin{equation*}
    \pi_{k,r}:=\frac{(N-2)!(M!)^2\Gamma(k+1)}{(M-r)!(M+r)!(r!)^2\Gamma(k+N-1)}\prod_{i=0}^{r-1}(\lambda_k-\lambda_i )
\end{equation*}
for $r\geq 1$.
Then we have
\begin{equation*}
    \pi_{k,0}\leq \pi_{2,0}=\frac{2}{N(N-1)}     
\end{equation*}
and
\begin{equation*}
    \pi_{k,1}\leq \pi_{2,1}=\frac{4(N+1)(N-2)}{N^2(N-1)}.     
\end{equation*}
\end{lemma}
\begin{proof}
The upper bound of $\pi_{k,0}$ follows because it is decreasing in $k$.

For $r=1$, we have
\begin{equation*}
    \frac{\pi_{k+1,1}}{\pi_{k,1}}=\frac{(k+1)^2(k+N)}{k(k+N-1)^2}<1
\end{equation*}
for $N \geq 6$, which yields the desired result.
\end{proof}

\begin{lemma}\label{cNrest}
Let $c_{N,0}:=\frac{N!}{(N-2)!}$ and
\begin{equation*}
    c_{N,r}:=\frac{N!(M-r)!(M+r)!(r!)^2}{(N-2)!(M!)^2}\sum_{d=0}^{M-r}\frac{(M-r-d+1)!}{(M+r+d+1)!}C_{r,d}. 
\end{equation*}
Then, we have for any $2\leq r\leq M$,
\begin{equation*}
    c_{N,r}\leq \frac{2N(N-1)}{9}. 
\end{equation*}

\end{lemma}
\begin{proof}
Recall that 
\begin{equation*}
    \begin{aligned}
        C_{r,d}&=\binom{M}{r}\binom{M-r}{d}(r)_d(M+r)_d\frac{3(M+r)+d}{3(M+r)}\\
        &=\frac{M!(r+d-1)!(M+r+d-1)!}{r!d!(M-r-d)!(r-1)!(M+r-1)!}\frac{3(M+r)+d}{3(M+r)}.
    \end{aligned}
\end{equation*}
Then we have
\begin{equation*}
    c_{N,r}=\frac{N(N-1)}{3}\sum_{d=0}^{M-r}\frac{\binom{r+d-1}{r-1}}{\binom{M}{r}}\frac{(M-r-d+1)(3M+3r+d)}{(M+r+d)(M+r+d+1)}.    
\end{equation*}

If $2\leq r=M=\frac{N}{2}-1$, then we have $N\geq 6$ and $c_{N,M}=N\leq \frac{2N(N-1)}{9}$. In the following, we will assume $2\leq r\leq M-1$.

Set $v=r+d$. We have
\begin{equation*}
    \frac{(M-r-d+1)(3M+3r+d)}{(M+r+d)(M+r+d+1)}=(M-v+1)\frac{(3M+2r+v)}{(M+v)(M+v+1)}=:(M-v+1)\varphi(v).   
\end{equation*}

Since
\begin{equation*}
    \begin{aligned}
        &\binom{v-1}{r-1}
        =\binom{v}{r}-\binom{v-1}{r}, \\   
        &v\binom{v-1}{r-1}=r\binom{v}{r}, \\   
        &d\binom{v-1}{r-1}=r\binom{v-1}{r},
    \end{aligned}    
\end{equation*}
we have
\begin{equation}\label{combsum}
    \begin{aligned}
        &\sum_{d=0}^{M-r}\frac{\binom{v-1}{r-1}}{\binom{M}{r}}=1,\\
        &\sum_{d=0}^{M-r}\frac{\binom{v-1}{r-1}}{\binom{M}{r}}(M-v+1)=\frac{M+1}{r+1},\\
        &\sum_{d=0}^{M-r}\frac{\binom{v-1}{r-1}}{\binom{M}{r}}(M-v+1)\frac{v-r}{M-r}=\frac{r(M+1)}{(r+1)(r+2)}.
    \end{aligned}
\end{equation}

Since $\varphi$ is a convex function on $[r,M]$, we have
\begin{equation*}
    \varphi(v)\leq \frac{M-v}{M-r}\varphi(r)+\frac{v-r}{M-r}\varphi(M).
\end{equation*}

Hence, combining with \eqref{combsum} and further computations, we have
\begin{equation*}
    \begin{aligned}
        \sum_{v=r}^{M}\frac{\binom{v-1}{r-1}}{\binom{M}{r}}(M-v+1)\varphi(v)
        &\leq  \frac{M+1}{r+1}\left(\frac{2}{r+2}\varphi(r)+\frac{r}{r+2}\varphi(M)\right)\\
        &=\frac{M+1}{(r+1)(r+2)}\left[\frac{6}{M+r+1}+\frac{r(2M+r)}{M(2M+1)}\right]\\
        &\leq  \frac{2}{r+1}.
    \end{aligned}
\end{equation*}

Consequently, we have
\begin{equation*}
    c_{N,r}=\frac{N(N-1)}{3}\sum_{v=r}^{M}\frac{\binom{v-1}{r-1}}{\binom{M}{r}}(M-v+1)\varphi(v)\leq \frac{2N(N-1)}{3(r+1)}\leq \frac{2N(N-1)}{9}
\end{equation*}
for any $2\leq r\leq M$.

\end{proof}

Finally, we can prove Proposition \ref{RNkest}.
\begin{proof}[Proof of Proposition \ref{RNkest}]
For $N=4$, we have
\begin{equation*}
    \frac{\Gamma(k+1)}{\Gamma(k+3)}R_{4,k}=4+\frac{16}{\lambda_k+2}\leq \frac{16}{3}=c_4
\end{equation*}
since $\lambda_k\geq \lambda_2=10$.

For $N\geq 6$, we write
\begin{equation*}
    \frac{\Gamma(k+1)}{\Gamma(k+N-1)}R_{N,k}=\sum_{r=0}^{\min\{k,M\}}\pi_{k,r}c_{N,r},
\end{equation*}
where $\pi_{k,r}$ and $c_{N,r}$ are defined in Lemma \ref{pikrest} and Lemma \ref{cNrest} respectively.

Note that
\begin{equation}\label{eq:pi-hyper}
    \pi_{k,r}=\frac{(-M)_r(-k)_r(k+2M+1)_r}{(M+1)_r(1)_r r!}\pi_{k,0}.
\end{equation}
We recall the Pfaff--Saalsch\"utz identity
\begin{equation}
 {}_3F_2\!\left(\begin{matrix}a,b,-n\\c,1+a+b-c-n\end{matrix};1\right)
 =\frac{(c-a)_n(c-b)_n}{(c)_n(c-a-b)_n},\qquad n\in \mathbb{N}.\label{eq:PS}
\end{equation}
Take
\[
 a=-M,\qquad b=k+2M+1,\qquad n=k,\qquad c=M+1.
\]
Then $1+a+b-c-n=1$, exactly matching \eqref{eq:pi-hyper}, and \eqref{eq:PS} gives
\begin{align*}
 \sum_{r=0}^{\min\{k,M\}}
 \frac{(-M)_r(-k)_r(k+2M+1)_r}{(M+1)_r(1)_r\,r!}
 &=\frac{(2M+1)_k(-M-k)_k}{(M+1)_k(-k)_k}\\
 &=\frac{(k+2M)!}{(2M)!k!}.
\end{align*}
Since
\[
 \pi_{k,0}=\frac{(2M)!k!}{(k+2M)!},
\]
we have proved the normalization
\begin{equation*}
 \sum_{r=0}^{\min\{k,M\}}\pi_{k,r}=1.
\end{equation*}

By Lemma \ref{cNrest} and Lemma \ref{pikrest}, we have
\begin{equation*}
    \begin{aligned}   
        \sum_{r=0}^{\min\{k,M\}}\pi_{k,r}c_{N,r}
        &\leq  \frac{2N(N-1)}{9}+\pi_{k,0}(N(N-1)-\frac{2N(N-1)}{9})+\pi_{k,1}(\frac{N(N-1)}{3}-\frac{2N(N-1)}{9})\\
        &\leq \frac{2(N^3+N^2+5N-4)}{9N}=c_N
    \end{aligned}     
\end{equation*}
for $N \geq 6$ as desired.
\end{proof}

\section*{Acknowledgements}
 The research of  C. Gui is supported by NSFC Key Program (Grant No.12531010), University of Macau research grants CPG2024-00016-FST, CPG2025-00032-FST, CPG2026-00027-FST, SRG2023-00011-FST, MYRGGRG2023-00139-FST-UMDF, UMDF Professorial Fellowship
of Mathematics, Macao SAR FDCT 0003/2023/RIA1 and Macao SAR FDCT 0024/2023/RIB1. The research of J. Wei is partially supported by General Research Grant of HKSAR GRF 14309824. The authors acknowledge the use of AI tools. All mathematical arguments and proofs in the final manuscript were checked and written by the authors.

\end{document}